\documentclass[12pt]{amsart}
\usepackage{amssymb,amsmath,amsfonts,latexsym, mathtools}
\usepackage{bm, enumerate}
\usepackage{graphicx}
\usepackage{color,soul}
\usepackage{hyperref}
\usepackage{dsfont}
\usepackage{bm}
\newcommand{\newsection}[1]{\setcounter{equation}{0} \section{#1}}
\newcommand{\bea}{\begin{eqnarray}}
\newcommand{\eea}{\end{eqnarray}}

\newcommand{\clb}{\mathcal{B}}

\newcommand{\cld}{\mathcal{D}}
\newcommand{\cle}{\mathcal{E}}
\newcommand{\clf}{\mathcal{F}}
\newcommand{\clg}{\mathcal{G}}
\newcommand{\clh}{\mathcal{H}}
\newcommand{\clk}{\mathcal{K}}

\newcommand{\clm}{\mathcal{M}}
\newcommand{\cln}{\mathcal{N}}

\newcommand{\clq}{\mathcal{Q}}

\newcommand{\cls}{\mathcal{S}}
\newcommand{\clt}{\mathcal{T}}

\newcommand{\D}{\mathbb{D}}

\newcommand{\Z}{\mathbb{Z}}
\newcommand{\z}{\bm{z}}

\def \qed {\hfill \vrule height6pt width 6pt depth 0pt}
\def\textmatrix#1&#2\\#3&#4\\{\bigl({#1 \atop #3}\ {#2 \atop #4}\bigr)}
\def\dispmatrix#1&#2\\#3&#4\\{\left({#1 \atop #3}\ {#2 \atop #4}\right)}
\newcommand{\be}{\begin{equation}}
\newcommand{\ee}{\end{equation}}
\newcommand{\ben}{\begin{eqnarray*}}
\newcommand{\een}{\end{eqnarray*}}

\newcommand{\bi}{\begin{itemize}}
\newcommand{\ei}{\end{itemize}}

\newcommand{\undertilde}[1]{\underset{\widetilde{}}{#1}}
\newcommand{\ut}{\undertilde}

\newtheorem{Theorem}{\sc Theorem}[section]
\newtheorem{Lemma}[Theorem]{\sc Lemma}
\newtheorem{Proposition}[Theorem]{\sc Proposition}
\newtheorem{Corollary}[Theorem]{\sc Corollary}
\newtheorem{Definition}[Theorem]{\sc Definition}
\newtheorem{Example}[Theorem]{\sc Example}

\newtheorem{Remark}[Theorem]{\sc Remark}
\newtheorem{Remarks}[Theorem]{\sc Remarks}
\newtheorem{Note}[Theorem]{\sc Note}
\newtheorem{Question}{\sc Question}
\newtheorem{ass}[Theorem]{\sc Assumption}
\newcommand{\bt}{\begin{Theorem}}
\def\beginlem{\begin{Lemma}}
\def\beginprop{\begin{Proposition}}
\def\begincor{\begin{Corollary}}
\def\begindef{\begin{Definition}}
\def\beginexamp{\begin{Example}}
\def\beginrem{\begin{Remark}}
\def\beginq{\begin{Question}}
\def\beginass{\begin{ass}}
\def\beginnote{\begin{Note}}
\newcommand{\et}{\end{Theorem}}
\def\endlem{\end{Lemma}}
\def\endprop{\end{Proposition}}
\def\endcor{\end{Corollary}}
\def\enddef{\end{Definition}}
\def\endexamp{\end{Example}}
\def\endrem{\end{Remark}}
\def\endq{\end{Question}}
\def\endass{\end{ass}}
\def\endnote{\end{Note}}

\begin{document}

\title[Isometric dilation and commutant lifting theorem]{Isometric dilation and Sarason's commutant lifting theorem in several variables}

\author[B.K. Das]{B. Krishna Das}
\address{Department of Mathematics, Indian Institute of Technology Bombay, Powai, Mumbai, 400076, India}
\email{bata436@gmail.com, dasb@math.iitb.ac.in}

\author[S. Panja]{Samir Panja}
\address{Department of Mathematics, Indian Institute of Technology Bombay, Powai, Mumbai, 400076, India}
\email{spanja@math.iitb.ac.in, panjasamir2020@gmail.com}

\subjclass[2020]{47A20, 47A13, 47A56, 47B32, 30H10} 
\keywords{Isometric dilation, commutant lifting theorem, von Neumann inequality, Hardy space, commuting contractions, commuting isometries, bounded analytic functions.}

\begin{abstract}
The article deals with isometric dilation and commutant lifting for a class of $n$-tuples $(n\ge 3)$ of commuting contractions. We show that operator tuples in the class dilate to tuples of commuting isometries of BCL type. As a consequence of such an explicit dilation, we show that their von Neumann inequality holds on a one dimensional variety of the closed unit polydisc. On the basis of such a dilation, we prove a commutant lifting theorem of Sarason's type by establishing that every commutant can be lifted to the dilation space in a commuting and norm preserving manner. This further leads us to find yet another class of $n$-tuples $(n\ge 3)$ of commuting contractions each of which possesses isometric dilation.

\end{abstract}

\maketitle

\newsection{Introduction}

In the last century, one of the elegant results in operator theory, due to Sz.-Nagy and Foias (\cite{NF}), is that any contraction acting on a Hilbert space dilates to an isometry. That is, for a contraction $T$ acting on a Hilbert space $\clh$, there exists an isometry $V$ acting on a Hilbert space $\clk$ such that $\clh \subseteq \clk$  and 
\[ T^m = P_{\clh}V^m|_{\clh} \quad (m\in \Z_+).\]     
Generalizing this, Ando (\cite{Ando}) proved that for any pair of commuting contractions $(T_1, T_2)$ on $\clh$, there exists a pair of commuting isometries $(V_1, V_2)$ on $\clk$  such that $\clh \subseteq \clk$ and 
\[ T_1^{m_1} T_2^{m_2} = P_{\clh}V_1^{m_1} V_2^{m_2}|_{\clh} \quad ((m_1, m_2)\in \mathbb{Z}^2_+).\] 
The pair $(V_1,V_2)$ is called an isometric dilation of $(T_1,T_2)$. Different proofs of the Ando dilation have surfaced recently, see \cite{MB, DS, BSS, Sau}. However, further generalization of Ando's result for an $n$-tuples of commuting contractions $(n\ge 3)$ fails and a counterexample 
was discovered in \cite{Par}. It is still not very clear why the result fails in general and it is essential to understand subtle differences between the case $n=2$ and $n>2$. 
As a matter of course, it is important to identify $n$-tuples of commuting contractions each of which has an isometric dilation. By an isometric dilation of an $n$-tuple of commuting contractions $(T_1,\dots,T_n)$ on $\clh$ we mean an $n$-tuple of commuting isometries $(V_1,\ldots, V_n)$ on $\clk$ such that $\clh \subseteq \clk$ and 
\[T_1^{\alpha_1}\cdots T_n^{\alpha_n} = P_{\clh}V_1^{\alpha_1} \cdots V_n^{\alpha_n}|_{\clh},\]
for all $(\alpha_1,\ldots ,\alpha_n)\in \mathbb{Z}^n_+$. A fair amount of research has been carried out to identify different classes of $n$-tuples of commuting contractions each of which possesses an isometric dilation. Often such a class is obtained by imposing certain positivity condition. For more details  see \cite{BD, BDHS, BDS, BDDP, Breh, CVPoly, CV} for the polydisc setting, and for the setting of the unit ball in $\mathbb C^n$, see \cite{MulVas, Arve}, and references therein. Also see \cite{AEM, AE} for some general domains in $\mathbb C^n$.

One of the main contributions of this article is the discovery of a new class of $n$-tuples of commuting contractions $(n\ge 3)$ each of which possesses an isometric dilation. Our dilation result emerged from the study of a commutant lifting theorem (CLT), a very close neighbour of dilation theory. Indeed, this article is as much a study of dilation theory as of CLT. Commutant lifting theorem was first studied in the context of Nevanlinna-Pick interpolation by D. Sarason in \cite{Sara}. Soon after, Sz-Nagy and Foias extended the result in a general setting and proved that if $V$ on $\clk$ is an isometric dilation of a contraction $T$ acting on $\clh$, and  $X$ is a contraction on $\clh$ such that $X T=T X$, then there exists a contraction $Y$ on $\clk$ satisfying
\[Y V=V Y,\  Y^*|_\clh=X^* \text{ and } \|Y\|=\|X\|.
\]
Because of the second identity above, $Y$ is called a {\em co-extension or lift} of $X$. This result is known as Sz-Nagy and Foias' commutant lifting theorem and it is known to be equivalent to Ando dilation (\cite{Parrott}). Multivariable analogue of Sz-Nagy and Foias' commutant lifting theorem does not hold in general. See \cite{MUl} and also see \cite{Arch, BLTT, BTV, KDT, BDB, DavLe, EP} for related results. The best result known in the polydisc setting is due to Ball, Li, Timotin and Trent (\cite{BLTT}), where the authors considered a class of $n$-tuples of commuting contractions along with their specific isometric dilations and found a necessary and sufficient condition for the commutant lifting theorem to hold. This shows that even if we restrict ourselves to a suitable class of $n$-tuples of commuting contractions and their isometric dilations, it may not always possible to lift every commutant in a commuting and norm preserving way. Therefore, it raises the following natural question: 

\vspace{0.1cm}
\textit{ 
Is there a class of $n$-tuples of commuting contractions so that each of the $n$-tuples has an isometric dilation and every commutant can be lifted to the isometric dilation space in a commuting and norm preserving way?}
\vspace{0.1cm}

\noindent Needless to say that, the class of $n$-tuples requires to be nice enough for the above question to have a positive answer. We do not know the existence of any such class in the literature. In this article, we answer this question affirmatively by exhibiting a class and prove an exact analogue of Sz-Nagy and Foias' commutant lifting theorem, which is an another main contribution of this article. As a consequence, this leads us to identify a class of commuting contractions each of which possesses an isometric dilation. 

We now proceed to describe some of our results. This requires to set few notations and terminologies. For a Hilbert space $\clh$, we denote by $\clb(\clh)$ the space of all bounded operators acting on $\clh$. The set of all $n$-tuples of commuting contractions on $\clh$ is denoted by $\clt^n(\clh)$, that is,
\[\clt^n(\clh)=\{(T_1,\ldots,T_n):T_i\in\clb(\clh), \|T_i\|\leq 1, T_i T_j=T_j T_i,1\leq i,j \leq n\}.\] For $T=(T_1, \ldots,T_n)\in\clt^n(\clh)$, we denote by $P_T$ the product contraction of $T$, that is, $P_T:=T_1 \cdots T_n$, and
for $A\in \clb(\clh)$, we denote by $C_A$ the completely positive map on $\clb(\clh)$ induced by $A$ defined as 
\[
C_A:\clb(\clh)\to \clb(\clh),\ X\mapsto AXA^*.
\]
 For any $\gamma=(\gamma_1. \ldots, \gamma_n) \in\Z^n_+ $ and $T=(T_1, \ldots,T_n)\in\clt^n(\clh)$,
we define
\[\Delta^\gamma_{T}:=(I-C_{T_1})^{\gamma_1} \cdots (I-C_{T_n})^{\gamma_n}(I_\clh).\]
With this notation, it is easy to see that for $i\neq j$,
\[\Delta^{\bm{e}_i+\bm{e}_j}_{T}=(I-C_{T_i})(I-C_{T_j})(I_\clh)=I-T_i T^*_i-T_j T^*_j+T_i T_j T^*_i T^*_j,\]
where for $i=1, \ldots, n$,
\[\bm{e}_i=(0, \ldots,0, \underbrace{1}_{i\text{th place}} 0, \ldots, 0).\]
A contraction $S$ acting on $\clh$ is pure if $S^{* n}\to 0$ in the strong operator topology (SOT).
The class of $n$-tuples of commuting contractions that we consider is defined as follows: 
\begin{align*}
\mathfrak{B}^n_0(\clh):=\big\{T\in \clt^n(\clh): \Delta^{\bm{e}_i+\bm{e}_j}_{T}=0\ \text{ for all } i\neq j \text{  and  } P_T \text{  is pure}\big\}.    
\end{align*}
For the base case $n=2$, if $(T_1, T_2)$ is a pair of commuting contractions on $\clh$ such that $T_1$ is a co-isometry and $T_2$ is a pure contraction then clearly $(T_1,T_2)\in \mathfrak{B}^2_{0}(\clh)$. In fact, it turns out that if $\clh$ is finite dimensional then an arbitrary element of $\mathfrak{B}^2_0(\clh)$ is a direct sum of pairs of the above form (see Proposition \ref{finite}). However, in infinite dimension, it is no longer the case. We have demonstrated this by finding an example in Section \ref{Sec5}. It can be shown without much difficulty that for any $T\in \mathfrak{B}^n_0(\clh)$ and for a subset $\{i_1, \ldots, i_k\}$ of $\{1, \ldots, n\}$, 
\begin{align}
\Delta^{\bm{e}_{i_1}+\cdots+\bm{e}_{i_k}}_{T}=0.
\end{align}
In other words, $\mathfrak{B}^n_0(\clh)$ is a subclass of Brehmer $n$-tuples on $\clh$ (see \cite{Breh}). Therefore, any element of $\mathfrak{B}^n_0(\clh)$ has an isometric dilation. However, we show that any $T\in \mathfrak{B}^n_0(\clh)$ has an isometric dilation of the form $(M_{\Phi_1},\dots, M_{\Phi_n})$ acting on the Hardy space $H^2_{\cle}(\mathbb D)$ for some co-efficient Hilbert space $\cle$, where 
\[
\Phi_i(z)=U_i (P_i^\perp+ z P_i) \,\, ( z\in \D)
\]
and for all $i=1,\dots,n$, $U_i\in \clb(\clh)$ is a unitary, $P_i\in \clb(\cle)$ is an orthogonal projection so that 
\begin{equation}\label{BCL}
\Phi_i(z)\Phi_j(z)=\Phi_j(z)\Phi_i(z)\ \text{ and } \Phi_1(z)\cdots\Phi_n(z)=zI_{\cle}\ \quad (z\in\mathbb D).
\end{equation}
Moreover, the tuple of isometries also belongs to the same class, that is, \[(M_{\Phi_1},\dots, M_{\Phi_n})\in \mathfrak{B}^n_0(H^2_{\cle}(\D)).\]  
In case the reader is wondering what properties are required for $U_i$ and $P_i$ so that ~\eqref{BCL} holds, see \cite{BCL, BDF} or Section \ref{Sec2} below. Here, the Hardy space $H^2_{\cle}(\D)$ is defined as 
 \[H^2_{\cle}(\D)=\{f(z)=\sum_{k=0}^\infty a_k z^k: a_k \in \cle,  z \in \D,\|f\|^2=\sum_{k=0}^\infty \|a_k\|^2 <\infty\},\]
and the isometry $M_{\Phi_i}:H^2_{\cle}(\D) \to H^2_{\cle}(\D)$ is the multiplication operator associated to $\Phi_i$ defined as 
\[M_{\Phi_i} f(z)=\Phi_i(z)f(z)\quad (f \in H^2_{\cle}(\D)).\]
It is evident from ~\eqref{BCL} that $(M_{\Phi_1},\dots, M_{\Phi_n})$ is an $n$-tuple of commuting isometries and $M_{\Phi_1}\cdots M_{\Phi_n}=M_z$, where $M_z$ is the shift on $H^2_{\cle}(\D)$. Thus, $M_z$ can be factorised as the product of the $n$-tuple $(M_{\Phi_1},\dots, M_{\Phi_n})$ and it is a remarkable result of Berger, Coburn and Lebow (\cite{BCL}) that any commuting contractive factor of $M_z$ arises in this way. For this reason, $(M_{\Phi_1},\cdots, M_{\Phi_n})$ is commonly known as a BCL $n$-tuple. Now coming back to the present context, the explicit isometric dilation of $T\in \mathfrak{B}^n_0(\clh)$ to a BCL $n$-tuple plays a pivotal role in what follows and aids to establish a sharp von Neumann type inequality. For a contraction $T\in \clb(\clh)$, the defect operator $D_T$ and the defect space $\cld_T$ of $T$ are defined as 
\[D_T:=(I-T T^*)^{1/2} \text{ and }\cld_T:= \overline{\text{ran}}\, D_T. \]
If an $n$-tuple of commuting contractions $T$ has an isometric dilation, then its von Neumann inequality holds on $\overline{\D}^n$ (see \cite{NF}), that is,
for all polynomials $p\in \mathbb{C}[z_1, \ldots, z_n]$,
\begin{align}\label{von_inequ}
    \|p(T)\|_{\clb(\clh)} \leq \sup_{\z\in \overline{\D}^n } |p(\z)|.
\end{align}
The above inequality is refined and a significant reduction in terms of the dimension of the set where the inequality holds is obtained.   
\begin{Theorem}
Let $T=(T_1,\ldots,T_n)\in \mathfrak{B}^n_0(\clh)$ and $\text{dim}\, \cld_{T_i}< \infty$ ($i=1,\ldots,n$). Then there exists a one-dimensional symmetric algebraic variety $W$ in $\mathbb{C}^n$ such that for all polynomials $p\in \mathbb{C}[z_1, \ldots, z_n]$,
\[\|p(T)\|_{\clb(\clh)} \leq \sup_{\z\in W\cap \overline{\D}^n} |p(\z)|.\]
\end{Theorem} 
This result is proved in Theorem~\ref{sharp vN inequality} below. Such a refined von Neumann inequality has been studied earlier for various other class of $n$-tuples of commuting contractions (see \cite{AgMc, BDHS, DS}) and most often such a result is obtained on the basis of finding explicit and nice enough isometric dilations.   

Using the explicit isometric dilation, as indicated above, we answer the question raised earlier and prove an analogue of Sz-Nagy and Foias' commutant lifting theorem for $\mathfrak{B}^n_0(\clh)$.

\begin{Theorem}\label{CLT--}
Let $T=(T_1,\ldots,T_n)\in \mathfrak{B}^n_0(\clh)$ and let $(M_{\Phi_1},\dots, M_{\Phi_n})\in \mathfrak{B}^n_0(H^2_{\cle}(\D))$ be an isometric dilation of $T$ for some Hilbert space $\cle$. If  $XT_i=T_iX$ ($i=1,\dots,n$) for some contraction $X\in \clb(\clh)$,  then there exists $Y\in \clb(H^2_{\cle}(\D))$ such that $Y$ is a lift of $X$, 
\[
YM_{\Phi_i}=M_{\Phi_i}Y \quad(i=1,\ldots,n), \text{ and } \|Y\|=\|X\|.
\]
\end{Theorem}
This result is proved in Theorem~\ref{Comm_lift_theo} below. The proof heavily relies on a detailed analysis on the structure of commutants of $T$ as well as the structure of commutants of its isometric dilation $(M_{\Phi_1},\dots, M_{\Phi_n})$. For instance, we show that if $T=(T_1,\ldots,T_n)\in \mathfrak{B}^n_0(\clh)$ and $XT_i=T_iX$ for some contraction $X\in \clb(\clh)$, then 
\[
 I-X X^*=G_1 +\cdots + G_n \text{ and }  (I-C_{T_i})(G_j)=
 G_j -T_i G_j T^*_i =0
 \quad (i\neq j),
\]
where $G_i$'s are positive operators on $\clh$ (see Theorem~\ref{Nessry_cond} below). The above identity has some resemblance to the necessary and sufficient condition for CLT obtained in ~\cite{BLTT}. In fact, Theorem~\ref{CLT--} is motivated from the above identity and the result of ~\cite{BLTT}. However, the proof of the theorem is entirely different from that of ~\cite{BLTT}, partly because the isometric dilation and the dilation map are different, and one can not directly use the result of ~\cite{BLTT} in this context as well. Finally, as CLT and dilation are closely related, the above theorem provides a new class of $n$-tuples of commuting contractions such that each of the tuples has an isometric dilation. The following result is proved in Theorem~\ref{main 2} below.

\begin{Theorem}
Let $(T_1,\dots,T_n, T_{n+1})$ be an $(n+1)$-tuple of commuting contractions on $\clh$ such that $(T_1,\dots,T_n)\in \mathfrak{B}^n_0(\clh)$. Then $(T_1,\dots,T_n, T_{n+1})$ has an isometric dilation and consequently, satisfies the von Neumann inequality. 
\end{Theorem}

 This article is organized as follows. In the next section, we discuss some preparatory results. Section \ref{Sec3} is devoted to find an explicit isometric dilations for $\mathfrak{B}^n_0(\clh)$ and to obtain a sharper version of the von Neumann inequality for operator $n$-tuples in $\mathfrak{B}^n_0(\clh)$ with finite dimensional defect spaces. In Section \ref{Sec4}, we prove CLT for $\mathfrak{B}^n_0(\clh)$ and, as a consequence, we find a new class of $n$-tuples of commuting contractions each of which possesses an isometric dilation. Examples on $n$-tuples in $\mathfrak{B}^n_0(\clh)$ are considered in Section \ref{Sec5}.

\newsection{Preparatory results}\label{Sec2}

We set few definitions and conventions and also establish some preliminary results, in this section, which we need subsequently.
 All Hilbert spaces are assumed to be separable and over the complex field $\mathbb{C}$.


\subsection{\textbf{Isometric dilation.}} 
 Let $T\in\clt^n(\clh)$. If there exist an $n$-tuple $V=(V_1, \ldots,V_n)$ of commuting isometries on a Hilbert space $\clk$ and an isometry  $\Pi:\clh \to \clk$ such that 
\[
\Pi T^*_i=V^*_i \Pi, \quad (i=1,\ldots,n)
\]
then 
\[(T_1,\ldots,T_n)\cong (P_{\clq}V_1|_{\clq}, \ldots,P_{\clq}V_n|_{\clq}),\]
where $\clq=\mbox{ran}\Pi$ is a joint $(V_1^*,\dots,V_n^*)$-invariant
subspace of $\clk$ and the unitary equivalence is induced by $\Pi$. Here, two tuples $(A_1, \ldots, A_n)\in \clt^n(\clh)$ and $(B_1, \ldots, B_n)\in \clt^n(\clk)$ is said to be unitarily equivalence, denoted by $ (A_1,\ldots,A_n)\cong (B_1,\ldots,B_n)$, if there exists a unitary $U:\clh \to \clk$ such that $UA_i=B_i U$ for all $i=1,\ldots,n$. Since $(V_1,\dots,V_n)$ is a co-extension of $(P_{\clq}V_1|_{\clq}, \ldots,P_{\clq}V_n|_{\clq})$ and co-extension is a special type of dilation, we say that $V$ is an isometric dilation of $T$. In such a case, the isometry $\Pi$ is called the dilation map. Moreover, $V$ is a \textit{minimal isometric dilation} of $T$ if 
\[
\clk= \bigvee_{\alpha\in \mathbb Z_+^n}V_1^{\alpha_1}\cdots V_n^{\alpha_n}(\clq).
\]Throughout the article, most of the isometric dilations we obtained are of the above form.

\subsection{\textbf{Commuting tuples in $\mathfrak{B}^n_0(\clh)$}}
 Let $T\in \clt^n(\clh)$, and let $P_T:=T_1 \cdots T_n$ be the product contraction of $T$. Recall that for all $k=2,\ldots, n$, and $\{i_1, \ldots, i_k\}\subseteq \{1, \ldots, n\}$,
\begin{align}
\Delta^{\bm{e}_{i_1}+\cdots+\bm{e}_{i_k}}_{T}=(I-C_{T_{i_1}})\cdots(I-C_{T_{i_k}})(I_\clh),
\end{align}
and 
\begin{align*}
\mathfrak{B}^n_0(\clh):=\big\{T\in\clt^n(\clh) :\Delta^{\bm{e}_i+\bm{e}_j}_{T}=0\ \text{ for all } i\neq j \text{  and  } P_T \text{  is pure}\big\}.    
\end{align*}
We consider couple of elementary properties of $\mathfrak{B}^n_0(\clh)$ in the next two lemmas. Both the lemmas remain valid even if we do not assume $P_T$ to be pure in the definition of $\mathfrak{B}^n_0(\clh)$. The first lemma shows that any element of $\mathfrak{B}^n_0(\clh)$ satisfy Brehmer positivity.
\begin{Lemma}\label{Breh_0}
If $T=(T_1, \ldots, T_n)\in \mathfrak{B}^n_0(\clh)$, then for all $k=3,\dots, n$ and $\{i_1, \ldots, i_k\}\subseteq\{1, \ldots, n\}$, 
\begin{align*}
\Delta^{\bm{e}_{i_1}+\cdots+\bm{e}_{i_k}}_{T}=0.
\end{align*}
\end{Lemma}
\begin{proof}
 Since $\Delta_{T}^{\bm{e}_i+\bm{e}_j}=0$ for all $i \neq j$, the proof of the lemma follows from the definition of $\Delta^{\bm{e}_{i_1}+\cdots+\bm{e}_{i_k}}_{T}$, that is,
\begin{align*}
    \Delta^{\bm{e}_{i_1}+\cdots+\bm{e}_{i_k}}_{T}&=(I-C_{T_{i_1}})\cdots(I-C_{T_{i_k}})(I_\clh)\\
    &=(I-C_{T_{i_1}})\cdots(I-C_{T_{i_{k-2}}})( \Delta^{\bm{e}_{i_{k-1}}+\bm{e}_{i_{k}}}_{T}).
\end{align*}
\end{proof}

\begin{Lemma}\label{S=0}
If $T=(T_1,\ldots, T_n)\in \mathfrak{B}^n_0(\clh)$,
then for any  $\emptyset\neq F\subseteq \{1, \ldots,i-1,i+1,\dots, n\}$,
\[
\Delta^{(1,1)}_{(T_i, T_F)}=0,
\] 
where for $F=\{n_1, \ldots, n_j\}\subseteq \{i_1, \ldots, i_k\}$, $T_F:= T_{n_1}\cdots T_{n_j}$.
\end{Lemma}
\begin{proof}
We prove the lemma only for $F=\{2,\dots,n\}$ as the proof for arbitrary $F$ is similar.  Using Agler's hereditary functional calculus (\cite{AG1}) to the following identity, for $z_1, \ldots, z_n \in \D$ and $w_1, \ldots, w_n\in \D$, 
\begin{align*}
     (1-z_1\Bar{w}_1)(1-z_2 \cdots z_n\Bar{w}_1\cdots \Bar{w}_n)
    &=(1-z_1 \Bar{w}_1)\Big(\sum_{i=2}^n (1-z_i \Bar{w}_i)-\sum^n_{\underset{i< j}{i,j=2}}(1-z_i \Bar{w}_i)(1-z_j \Bar{w}_j)\\
    +\sum_{\underset{i< j< k}{i,j,k =2}}^n (1-z_i \Bar{w}_i)(1-z_j \Bar{w}_j)&(1-z_k \Bar{w}_k)- \cdots (-1)^{n-2} (1-z_2 \Bar{w}_2) (1-z_3 \Bar{w}_3)\cdots (1-z_n \Bar{w}_n)\Big),
\end{align*}
we obtain
\begin{align*}
    \Delta^{(1,1)}_{(T_i, T_F)}= \sum_{i=2}^n \Delta_{T}^{\bm{e}_1+\bm{e}_i}-\sum^n_{\underset{i< j}{i,j=2}} \Delta_{T}^{\bm{e}_1+\bm{e}_i+\bm{e}_j}+ \sum_{\underset{i< j< k}{i,j,k =2}}^n \Delta_{T}^{\bm{e}_1+\bm{e}_i+\bm{e}_j+\bm{e}_k}- \cdots (-1)^{n-2}  \Delta_{T}^{\bm{e}_1+\cdots+\bm{e}_n}.
\end{align*}
The proof now follows by Lemma~\ref{Breh_0}.
\end{proof}

\subsection{\textbf{BCL tuple.}}
Let $\cle$ be a Hilbert space. The space of all bounded $\clb(\cle)$-valued holomorphic functions on $\D$ is denoted by $H^\infty_{\clb(\cle)}(\D)$. For any $\Psi\in H^\infty_{\clb(\cle)}(\D)$ the associated multiplication operator $M_{\Psi}$ is a bounded operator on the Hardy space $H^2_{\cle}(\D)$ and defined as
\[M_{\Psi} f(z)=\Psi(z)f(z)\quad (f \in H^2_{\cle}(\D)).\]
A BCL $n$-tuple on $H^2_{\cle}(\D)$ is an $n$-tuple of commuting isometries $(M_{\Phi_1},\dots, M_{\Phi_n})$ on $H^2_{\cle}(\D)$ such that $M_{\Phi_1}\cdots M_{\Phi_n}=M_z$, where $M_z$ is the shift on $H^2_{\cle}(\D)$. For a unitary $U$ on $\cle$ and an orthogonal projection $P\in\clb(\cle)$, it is easy to verify that $M_\Phi$ is an isometry, where 
\begin{equation}\label{BCL1}
\Phi(z)=U (P^\perp+ z P) \,\, ( z\in \D).
\end{equation}
The Berger, Coburn and Lebow's theorem shows that any BCL $n$-tuple arises in the same way as ~\eqref{BCL1}. We recall the theorem for further use.

\begin{Theorem}[cf. \cite{BCL}]\label{BCL main}
Let $\Phi_1, \ldots, \Phi_n$ be $\clb(\cle)$-valued bounded analytic function on $\mathbb D$ for some Hilbert space $\cle$. Then $(M_{\Phi_1},\dots, M_{\Phi_n})$ on $H^2_{\cle}(\D)$ is a BCL $n$-tuple if and only if
for all $i=1,\dots,n$,
\[
\Phi_i(z)=U_i (P_i^\perp+ z P_i) \,\, ( z\in \D)
\]
for some unitaries $U_i$'s on $\cle$ and orthogonal projections $P_i$'s in $\clb(\cle)$ satisfying
\begin{itemize}
     \item[(1)] $U_i U_j=U_j U_i$ ($i, j=1,\ldots,n$);
    \item[(2)] $U_1 \cdots U_n =I_\cle$;
     \item[(3)]  $P_i + U^*_i P_j U_i=P_j + U^*_j P_i U_j \leq I_\cle (i \neq j)$; and
    \item[(4)] $P_1 +U^*_1 P_2 U_1 + \dots + U^*_1 \cdots U^*_{n-1} P_n U_1 \cdots U_{n-1}=I_\cle$.
\end{itemize}
\end{Theorem}
It is evident from the above theorem that a BCL $n$-tuple is completely determined by a triple $(\cle, \undertilde{U}, \undertilde{P})$, where $\cle$ is a Hilbert space, $\undertilde{U}=(U_1,\dots,U_n)$ is an $n$-tuple of unitaries on $\cle$, $\undertilde{P}=(P_1,\dots,P_n)$ is an $n$-tuple of projections in $\clb(\cle)$ and $U_i$'s and $P_i$'s satisfy conditions $(1)- (4)$ in the above theorem. \textit{For the rest of the article,} we use a triple $(\cle, \undertilde{U}, \undertilde{P})$ to refer to a BCL $n$-tuple and it would always inherit the meaning that the commuting unitary tuple $\ut{U}=(U_1,\dots, U_n)$ and the projection tuple $\ut{P}=(P_1,\dots,P_n)$ satisfy conditions $(1)- (4)$, and the corresponding BCL $n$-tuple is given by $\Phi_i(z)=U_i (P_i^\perp+ z P_i) \,\, ( z\in \D, i=1,\dots,n)$.  Few remarks are in order.
\begin{Remarks}\label{Remark_BCL}
\begin{itemize}
\item[(i)] Condition \textup{$(3)$} in the above theorem implies mutual orthogonality of the projections $P_i$ and $U^*_i P_j U_i$ ($i\neq j$). That is, 
\[P_i (U^*_i P_j U_i)=(U^*_i P_j U_i) P_i=0,\]
which is also equivalent to 
\[
(U_iP_iU_i^*)P_j= P_j(U_iP_iU_i^*)=0.
\]
 \item[(ii)] Conditions \textup{$(2)$} and \textup{$(4)$}, in the above theorem, together imply that
    \[M_{\Phi_1}\cdots M_{\Phi_n} =M_z.\]
\item[(iii)]
For $n=2$, if $(M_{\Phi_1}, M_{\Phi_2})$ is a BCL pair corresponding to $(\cle, \ut{U},\ut{P})$. Then $U_2$ and $P_2$ can be recovered from $U_1$ and $P_1$. Indeed, by condition \textup{$(2)$}, $U_2=U_1^*$ and by condition \textup{$(4)$}, $P_2=I-U_1 P_1 U^*_1$. Thus, the BCL pair $(M_{\Phi_1}, M_{\Phi_2})$ is completely determined by $(\cle, U, P)$, where $U=U_1$ and $P=I-U_1 P_1 U^*_1$, and consequently, $\Phi_1$ and $\Phi_2$ can be written as 
\[\Phi_1(z)=(P+ z P^\perp)U \,\, \textit{and}\,\, \Phi_2(z)=U^*(P^\perp+ z P)\quad ( z\in \D).\]

\end{itemize}
\end{Remarks}

\begin{Lemma}\label{Note_BCL}
Let $(M_{\Phi_1}, \ldots, M_{\Phi_n})$ be a BCL $n$-tuple corresponding to $(\cle,\ut{U},\ut{P})$. Then for all $i=1,\dots,n$, the orthogonal projection $(I_{H^2_{\cle}(\D)}-M_{\Phi_{i}}M^*_{\Phi_{i}})$ has the block matrix representation
\[
\begin{bmatrix}
U_iP_iU_i^* & 0\\
0& 0
\end{bmatrix}: \cle\oplus zH^2_{\cle}(\D) \to \cle\oplus zH^2_{\cle}(\D)
\]
with respect to the decomposition $H^2_{\cle}(\D)=\cle\oplus zH^2_{\cle}(\D)$.
\end{Lemma}
\begin{proof}
Since $M_{\Phi_i}$ is an isometry, $(I_{H^2_{\cle}(\D)}-M_{\Phi_{i}}M^*_{\Phi_{i}})$ is a projection onto $H^2_{\cle}(\D)\ominus M_{\Phi_{i}}H^2_{\cle}(\D)$. Also since $M_z=M_{\Phi_1}\cdots M_{\Phi_n}$, $zH^2_{\cle}(\D)\subseteq \text{Ran} M_{\Phi_i}$, and therefore \[(I_{H^2_{\cle}(\D)}-M_{\Phi_{i}}M^*_{\Phi_{i}})(z H^2_{\cle}(\D) )=0.\]
On the other hand, for $\eta\in\cle$
\begin{align*}
(I_{H^2_{\cle}(\D)}-M_{\Phi_i}M_{\phi_i}^*)\eta =\eta-(U_iP_i^{\perp}+ zU_iP_i)(P_i^{\perp}U_i^*\eta)=(I_{\cle}-U_iP_i^{\perp}U_i^*)\eta=U_i P_i U^*_i\eta.
\end{align*}
This completes the proof.
 \end{proof}

Let $(M_{\Phi_1}, \ldots, M_{\Phi_n})$ be a BCL $n$-tuple corresponding to $(\cle,\ut{U},\ut{P})$. Suppose that $U_i P_j=P_j U_i$ for all $i,j=1,\dots,n$.
Then it follows from the condition \textup{$(4)$} in Theorem~\ref{BCL main} that 
\[P_1+P_2+\cdots+ P_n=I_{\cle}.\] In other words, $P_i$'s are mutually orthogonal and $\cle=\oplus_{i=1}^n \cle_i$, where $\cle_i=\text{ran}\, P_i$ $(i=1,\ldots,n$). Since $U_i P_j=P_j U_i$, therefore, $\cle_i$ is an $U_j$-reducing subspace of $\cle$.  Set $U^{(i)}_j:= U_j|_{\cle_i}$ for all $i,j=1,\dots,n$. Then it is clear that with respect to the decomposition $H^2_{\cle}(\D)=H^2_{\cle_1}(\D)\oplus\cdots \oplus H^2_{\cle_n}(\D)$, $M_{\Phi_i}$ has the following block diagonal representation:
\begin{equation}\label{BCL_Matrix}
M_{\Phi_i}=\begin{bmatrix}
I_{H^2(\D)}\otimes U^{(1)}_i & 0& \dots &0&\dots&0\\
0& I_{H^2(\D)}\otimes U^{(2)}_i &\dots &0 &\dots &0\\
\vdots &\vdots &\ddots &\vdots &\ddots &\vdots \\
0 &0&\dots &M_z\otimes U^{(i)}_i &\dots &0\\
\vdots &\vdots &\ddots & \vdots &\ddots &\vdots\\
0 &0 &\dots &0 &\dots &  I_{H^2(\D)}\otimes U^{(n)}_i
\end{bmatrix},
\end{equation}
for all $i=1,\dots,n$. The above block diagonal representation will be used many times in the sequel as most of the the BCL $n$-tuples, in this paper, will have the property that $U_i P_j=P_j U_i$ for all $i,j=1,\dots,n$. Finally, it follows from the commutativity of $U_i$'s that $(M_{\Phi_1}, \ldots, M_{\Phi_n})$ is doubly commuting, that is, $M_{\Phi_i}^* M_{\phi_j}= M_{\Phi_j}M_{\Phi_i}^*$ for all $i\neq j$. We summarise this in the next lemma for future reference.


\begin{Lemma}\label{BCL:DC}
Let $(M_{\Phi_1}, \ldots, M_{\Phi_n})$ be a BCL $n$-tuple corresponding to $(\cle,\ut{U},\ut{P})$ such that $U_i P_j=P_j U_i$ for all $i,j=1,\ldots,n$. Then $(M_{\Phi_1}, \ldots, M_{\Phi_n})$ is doubly commuting. 
\end{Lemma}

We now characterize BCL $n$-tuples on $H^2_{\cle}(\D)$ which are also in $\mathfrak{B}^n_0(H^2_{\cle}(\D))$. The result is proved for the case $n=2$ in \cite[Proposition 3.1]{HQY} and we generalize it here for arbitrary $n$.
\begin{Proposition}\label{main_Characterization}
Let $M=(M_{\Phi_1}, \ldots, M_{\Phi_n})$ on $H^2_{\cle}(\D)$ be a BCL $n$-tuple corresponding to $(\cle,\ut{U}, \ut{P})$. Then the following are equivalent. 
\begin{itemize}
\item[(i)] $M\in \mathfrak{B}^n_0(H^2_{\cle}(\D))$;
\item[(ii)] $U_i P_j=P_j U_i$ for $i, j =1,\ldots, n$.
    
\end{itemize}
\end{Proposition}
\textit{Proof of} $(i)\implies(ii).$ 
For a fix $i$, set  $\Tilde{M}_{\Phi_{i}}=M_{\Phi_{1}} \dots M_{\Phi_{i-1}} M_{\Phi_{i+1}} \dots M_{\Phi_{n}}$. Then, $(M_{\Phi_i}, \Tilde{M}_{\Phi_{i}})$ is a BCL pair and by Lemma \ref{S=0},
$\Delta^{(1,1)}_{(M_{\phi_i}\Tilde{M}_{\Phi_{i}})}=0$.
 
Therefore, by applying \cite[Proposition 3.1]{HQY} for the BCL pair $(M_{\Phi_i}, \Tilde{M}_{\Phi_{i}})$, we get $U_i P_i =P_i U_i$ for all $i=1,\ldots,n$.  
Now, for $i\neq j$, it follows from Lemma\ref{Note_BCL} that for $\eta \in \cle$,
\begin{align*}
&\Delta_{M}^{\bm{e}_i+\bm{e}_j} \eta\\
&=\Big ((I_{H^2_\cle(\D)}- M_{\Phi_i} M^*_{\Phi_i})-M_{\Phi_j}(I_{H^2_\cle(\D)}- M_{\Phi_i} M^*_{\Phi_i}) M^*_{\Phi_j}\Big )\eta\\
&= U_i P_i U^*_i\eta-U_j (P_j^\perp+M_z P_j)(I_{H^2_\cle(\D)}- M_{\Phi_i} M^*_{\Phi_i}) (P_j^\perp+M^*_z P_j)U^*_j\eta\\
&=U_i P_i U^*_i\eta-U_j (P_j^\perp+M_z P_j)U_i P_i U^*_i (P_j^\perp U^*_j \eta) \quad (\text{by Lemma~\ref{Note_BCL}})\\
&=U_i P_i U^*_i\eta-U_j U_i P_i U^*_i U^*_j\eta \quad (\text{as}\,\, U_i P_i U^*_i \perp P_j \text{ by part (i) of Remarks~\ref{Remark_BCL} }).
\end{align*} 
Since $\Delta_{M}^{\bm{e}_i+\bm{e}_j}=0$, the above computation shows that $P_j- U_i P_j U^*_i=0$. That is, $U_i P_j=P_j U_i$ for all $i\neq j$.
This completes the proof.

\textit{Proof of} $(ii)\implies(i).$
By Lemma~\ref{BCL:DC}, $M=(M_{\Phi_1}, \ldots, M_{\Phi_n})$ is doubly commuting. Then for any $i\neq j$, 
\[
\Delta_{M}^{\bm{e}_i+\bm{e}_j}=(I_{H^2_\cle(\D)}- M_{\Phi_i} M^*_{\Phi_i}) (I_{H^2_\cle(\D)}- M_{\Phi_j} M^*_{\Phi_j}). 
\]
We have already observed in the proof of Lemma~\ref{BCL:DC} that in this case  $P_i$'s are mutually orthogonal projections. The proof now follows from Lemma~\ref{Note_BCL}.
\qed

It has been observed in \cite[Proposition 3.6]{HQY} that for a BCL pair $(M_{\Phi_1}, M_{\Phi_2})$ on $H^2_{\cle}(\D)$, if $\cle$ is finite dimensional, then $\Delta^{(1,1)}_{ (M_{\phi_1}, M_{\phi_2})}\geq0$ is equivalent to $\Delta^{(1,1)}_{ (M_{\phi_1}, M_{\phi_2})}=0$. We also prove a similar result for BCL $n$-tuples and which shows that in the case of finite dimensional co-efficient space $\cle$, the class $\mathfrak{B}^n_0$ is more relevant.

\begin{Proposition}\label{Finite_dim}
Let $M=(M_{\Phi_1}, \ldots, M_{\Phi_n})$ be a BCL $n$-tuple corresponding to $(\cle, \ut{U},\ut{P})$ such that for $i\neq j$, $\Delta_{M}^{\bm{e}_i+\bm{e}_j} \geq0$. If $\text{dim}\, \cle<\infty$, then $M\in \mathfrak{B}^n_0(H^2_{\cle}(\D))$.   
\end{Proposition}
\begin{proof}
As we have seen in the proof of the above proposition, for $\eta \in \cle$ and $i\neq j$,
\begin{align*}
&\Delta_{M}^{\bm{e}_i+\bm{e}_j} \eta=(U_i P_i U^*_i-U_j U_i P_i U^*_i U_j^*)\eta.
\end{align*}
Then, $\Delta_{M}^{\bm{e}_i+\bm{e}_j} \geq 0$ ($i\neq j$) implies 
$U_i P_i U^*_i-U_j U_i P_i U^*_i U^*_j\geq 0$, that is, $U_j P_i U^*_j\leq P_i$.
Hence, $\text{ran}\, P_i$ is a $U_j$-invariant subspace. Since $\text{dim}\, \cle<\infty$, $\text{ran}\, P_i$ is also a $U_j$-reducing subspace. In other words, $U_j P_i=P_i U_j$ ($i\neq j$). Therefore, by the above computation, $ \Delta_{M}^{\bm{e}_i+\bm{e}_j}|_{\cle}= 0$ ($i\neq j$). On the other hand, it can be checked very easily that \[ \Delta_{M}^{\bm{e}_i+\bm{e}_j}|_{zH^2_{\cle}(\D)}=0\] even when $\cle$ is infinite dimensional. This completes the proof.
\end{proof}

We end the section with a result which ensures that BCL $n$-tuples are crucial while studying isometric dilations of $n$-tuples of commuting contractions. 
Observe that if $T=(T_1,\dots,T_n)\in\clt^n(\clh)$, then the product contraction $P_T=T_1\cdots T_n$ is pure if and only if $T^{* \alpha}\to 0$ in SOT as $\alpha\to \infty$. Here for $\alpha=(\alpha_1, \ldots, \alpha_n)\in \mathbb{Z}_+^n$, $T^{\alpha}=T_1^{\alpha_1} \cdots T_n^{\alpha_n}$ and $\alpha\to \infty$ means that $\alpha_i \to \infty$ for each $i=1,\ldots, n$.  Indeed, it follows from the following inequalities 
\[P_T^l P_T^{* l} \leq T^\alpha T^{* \alpha} \leq P_T^m P_T^{* m},\]
 where $m= \text{min}\,\{\alpha_i:i=1,\ldots, n\}$ and $l=\text{max}\,\{\alpha_i:i=1,\ldots, n\}$. 

\begin{Proposition}
Let $T=(T_1,\dots, T_n)\in \clt^n(\clh)$ be such that the product contraction $P_T$ is pure and $T$ admits an isometric dilation. Then any minimal isometric dilation of $T$ is unitarily equivalent to a BCL $n$-tuple $(M_{\Phi_1}, \ldots, M_{\Phi_n})$ on $H^2_{\cle}(\D)$ for some Hilbert space $\cle$.
\end{Proposition}
\begin{proof}
Let $V=(V_1, \ldots,V_n)$ on $\clk$ be a minimal isometric dilation of $T$. That is,
\[T^\alpha= P_\clh V^\alpha|_{\clh}\quad (\alpha \in \Z_+^n)\, \, \text{and}\,\,  \clk =\bigvee_{\beta \in \Z_+^n} V^\beta(\clh).\]
Since an $n$-tuple of commuting isometries with the product isometry being pure is unitarily equivalent to a BCL $n$-tuple \cite[Theorem 3.1]{BCL}, it suffices to prove that the product isometry of $V$ is pure or equivalently, by the observation made prior to this proposition, $V^{* \alpha}\to 0$ in SOT as $\alpha\to \infty$.
Since the product contraction $P_T:=T_1\cdots T_n$ is pure and the minimal isometric dilation is always a co-extension (\cite{NF}), we have for $h\in\clh$, 
\[\|V^{* \alpha} h\|=\|T^{* \alpha} h\| \to 0 \,\, \text{in SOT} \quad \text{as}\quad \alpha \to \infty.\]
Let $h\in \clk$ be such that $h=\sum_{0\le \beta \le \gamma} V^\beta h_\beta$ ($\gamma\in \mathbb Z_+^n, h_\beta\in \clh$).
Then for $\alpha \geq \gamma$, we have
\[V^{* \alpha} h=\sum_{ 0\le \beta\le \gamma} V^{*( \alpha-\beta)} h_\beta,\]
 and,
\[\|V^{* \alpha} h\|=\|\sum_{0 \le \beta \le \gamma} V^{*( \alpha-\beta)} h_\beta\| \leq \sum_{0 \le \beta \le \gamma} \|V^{*(\alpha-\beta)} h_\beta\|\to 0 \,\, \text{in SOT as } \alpha\to\infty.\]
Since $\{\sum_{0\le \beta \le \gamma} V^\beta h_\beta :\gamma\in \mathbb Z_+^n, h_\beta\in \clh\}$ is a dense subspace of $\clk$ and $\{V^{*\alpha}\}_{\alpha\in \mathbb Z_+^n}$ is a uniformly bounded family, it follows that $V^{* \alpha}\to 0$ in SOT as $\alpha\to \infty$. This completes the proof. 
\end{proof}



\newsection{Isometric dilation and von Neumann inequality for $\mathfrak{B}^n_0(\clh)$  }\label{Sec3}
This section is devoted to find explicit BCL type isometric dilations for tuples in $\mathfrak{B}^n_0(\clh)$. Isometric dilations of such tuples are known due to a result of Brehmer (\cite{Breh}). However, we find explicit BCL type dilations of such tuples and prove a sharp von Neumann inequality in certain cases. Explicit BCL type dilations are the key to prove the commutant lifting theorem in the subsequent section. For a contraction $T\in \clb(\clh)$, recall that the defect operator and the defect space of $T$ are defined as
\[D_T:=(I-T T^*)^{1/2} \quad \text{and}\quad \cld_T:= \overline{\text{ran}}\, D_T.\]
For an $n$-tuple of commuting contractions $T=(T_1,\ldots,T_n)\in\clt^n(\clh)$, 
the product contraction is denoted by $P_T:=T_1 \cdots T_n$. 
We are now ready to prove the dilation result for the case when the defect spaces of $T_i$'s are finite dimensional.
\begin{Theorem}\label{Dilion_finite}
Let $T=(T_1,\ldots,T_n)\in \mathfrak{B}^n_0(\clh)$ be such that $\text{dim}\, \cld_{T_i}< \infty$ ($i=1,\ldots,n$). Let $\cle=\bigoplus^n_{i=1} \cld_{T_i}$. Then there exist an isometry $\Pi: \clh \to H^2_\cle(\D)$, and a BCL $n$-tuple $(M_{\Phi_1}, \ldots,  M_{\Phi_n})\in \mathfrak{B}^n_0(H^2_{\cle}(\D))$ 
such that 
\[\Pi T^*_i=M_{\Phi_i}^* \Pi\]
for all $i=1,\dots,n$.
\end{Theorem}

 \begin{proof}
 Since $\Delta_{T}^{\bm{e}_i+\bm{e}_j}=0$ for all $i \neq j$,
 \begin{align*}
     D^2_{T_j}-T_i D^2_{T_j} T^*_i=D^2_{T_i}-T_j D^2_{T_i} T^*_j=0 \quad (i \neq j).
 \end{align*}
 As the defect spaces are assumed to be finite dimensional, the above identities induce  unitaries $W_j^{(i)}: \cld_{T_j} \to \cld_{T_j}$ for all $i\neq j$, defined by 
 \begin{align}\label{W_ij}
     W_j^{(i)}(D_{T_j}h)=D_{T_j} T^*_i h\quad (h\in\clh).
 \end{align}
Then observe that for $i, k=1,\ldots,n$ and $i\neq k\neq j$, $W_j^{(i)} W_j^{(k)}=W_j^{(k)} W_j^{(i)}$. Indeed,
\begin{align*}
   W_j^{(i)} W_j^{(k)} (D_{T_j} h)
 =W_j^{(i)} (D_{T_j} T^*_{k}h)
= D_{T_j} T^*_i T^*_{k} h
=W_j^{(k)} W_j^{(i)} (D_{T_j} h).
\end{align*}
Now for $i=j$, we set $W^{(j)}_j:=W_j^{ *(1)}\cdots W_j^{ *(j-1)} W_j^{* (j+1)}\cdots W_j^{* (n)}$. Then it is clear from the definition that
\[ W_j^{(j)}(D_{T_j} P_T^* h)=D_{T_j} T^*_j h\quad (h\in\clh),\]
and for each $j=1,\dots,n$, the tuple $(W_j^{(1)},\dots, W_j^{(n)})$ is a tuple of commuting unitaries and 
\begin{align}\label{W_ij=Id}
    W_j^{(1)}\cdots W_j^{(n)}=I_{\cld_{T_j}}
\end{align}
We now have all the ingredients to define the triple $(\cle,\ut{U},\ut{P})$ corresponding to the dilating BCL $n$-tuple. We take $\cle=\bigoplus^n_{i=1} \cld_{T_i}$, $\ut{U}=(U_1,\dots, U_n)$ and $\ut{P}=(P_1,\dots, P_n)$, where for $i=1,\dots,n$, 
\begin{align*}
    U^*_i=\bigoplus^n_{j=1} W_j^{(i)} \quad \text{and} \quad  P_i (x_1, \ldots, x_n)=(0,\ldots 0,  x_i , 0, \ldots 0)\quad (x_l \in \cld_{T_l}, l=1,\ldots,n).
\end{align*} 
By the construction of $U_i$ and $P_i$ and the identity \ref{W_ij=Id}, it easy to observe that for $i,j=1,\ldots,n$, 
\[U_i U_j=U_j U_i, \quad U_i P_j=P_j U_i,\quad P_1+ \ldots+P_n=I_\cle, \quad \text{and} \quad U_1 \cdots U_n=I_{\cle}.\]
Thus $\ut{U}$ and $\ut{P}$ satisfy all the conditions (1)-(4) as in Theorem~\ref{BCL main}, and by Proposition~\ref{main_Characterization}, the corresponding BCL $n$-tuple $(M_{\Phi_1},\ldots,M_{\Phi_n})\in \mathfrak{B}^n_0(H^2_{\cle}(\D))$.
In the rest of the proof, we construct the dilation map and show that $(M_{\Phi_1},\ldots,M_{\Phi_n})$ is an isometric dilation of $(T_1,\dots, T_n)$.
The following identity 
\begin{align*}
    D_{P_T}^2=D_{T_1}^2+T_1 D^2_{T_2}T^*_1+T_1 T_2 D^2_{T_3}T^*_1 T^*_2 +\ldots + T_1 \cdots T_{n-1} D^2_{T_n}T^*_1 \cdots  T^*_{n-1} h \quad (h\in\clh),
\end{align*}
implies that the map $\iota:\cld_{P_T} \to \cle$, defined by
\begin{align}\label{Iso_V}
    D_{P_T} h \mapsto (D_{T_1} h, D_{T_2} T^*_1 h, D_{T_3}T^*_1 T^*_2 h, \ldots, D_{T_n}T^*_1 \cdots T^*_{n-1} h),
\end{align}
 is an isometry. Since $P_T$ is a pure contraction, the map $\pi :\clh \to H^2_{\cld_{P_T}}(\D)$, defined by $\pi h=\sum_{k \in\Z_+}  z^k (D_{P_T} P_T^{* k}h) $ ($z\in \D, h\in \clh$), is an isometry. In fact, $\pi$ is the well-known dilation map for $P_T$. Finally, the dilation map for us is the isometry $\Pi:=(I \otimes \iota)\pi: \clh\to H^2_{\cle}(\D)$. To complete the proof, we show that for all $i=1,\dots,n$, 
\[\Pi T^*_i=M_{\Phi_i}^* \Pi.\]
Indeed, it is established by the rigorous computation below.  
 For $ k\in \Z_+, \eta\in\cle, h\in \clh$,
\begin{align*}
    &\langle M_{\Phi_i}^* \Pi h, z^k \eta\rangle\\
    &= \langle  (I \otimes \iota)\sum_{m \in\Z_+}  z^m (D_{P_T} P_T^{* m}h), (U_i P_i^\perp+z U_i P_i)z^k \eta \rangle \\
    &= \langle (P_i^\perp U_i^* \iota D_{P_T} P_T^{* k}h + P_i U_i^* \iota D_{P_T} P_T^{* k+1}h), \eta\rangle \\
   &= \langle P_i^\perp U_i^* (D_{T_1}P_T^{* k}h , D_{T_2} T^*_1 P_T^{* k}h,\ldots, D_{T_n}T^*_1 \cdots T^*_{n-1}P_T^{* k}h ), \eta\rangle \\
   &\quad \quad \quad \quad \quad \quad \quad \quad\quad \quad\quad+ \langle P_i U_i^* (D_{T_1} P_T^{* k+1}h, D_{T_2} T^*_1 P_T^{* k+1}h, \ldots,  D_{T_n}T^*_1 \cdots T^*_{n-1} P_T^{* k+1}h ), \eta\rangle \\
    &= \langle  U_i^* (D_{T_1}P_T^{* k}h,\ldots,  D_{T_{i-1}} T^*_1 \cdots, T^*_{i-2}P_T^{* k}h, 0,  D_{T_{i+1}} T^*_1 \cdots T^*_{i}P_T^{* k}h, \ldots, D_{T_n}T^*_1 \cdots T^*_{n-1}P_T^{* k}h) , \eta\rangle \\
   &\quad \quad \quad \quad \quad \quad \quad \quad \quad \quad+U_i^* (0,\cdots, 0, D_{T_i} T^*_1 \cdots, T^*_{i-1} P_T^{* k+1}h, 0, \ldots,  0), \eta\rangle \text{  (as  } U_i P_i=P_i U_i)\\
   &=\langle (D_{T_1} T_i^* P_T^{* k}h,\ldots,  D_{T_n}T^*_1 \cdots T^*_{n-1} T^*_i P_T^{* k}h ), \eta\rangle.
\end{align*}
On the other hand, for $\eta\in\cle, h\in \clh, k\in \Z_+$,
\begin{align*}
   \langle \Pi T^*_i h, z^k \eta\rangle &= \langle  (I \otimes \iota)\sum_{m \in\Z_+}  z^m (D_{P_T} P_T^{* m}T^*_ih), z^k \eta\rangle \\
   &= \langle  \iota D_{P_T} P_T^{* k}T^*_ih, \eta\rangle \\
   &=\langle (D_{T_1} T_i^* P_T^{* k}h,\ldots,  D_{T_n}T^*_1 \cdots T^*_{n-1} T^*_i P_T^{* k}h) , \eta\rangle. 
   \end{align*}
Hence, $\Pi T^*_i=M_{\Phi_i}^* \Pi\,\, (i=1,\ldots,n)$ as asserted. 
\end{proof}

An isometric dilation for arbitrary $T=(T_1,\ldots, T_n)\in \mathfrak{B}^n_0(\clh)$ can also be obtained in a similar way. Since the defect space $\cld_{T_j}$ may not be necessarily finite dimensional, the operator $W_j^{(i)}$ $(j\neq i)$, as defined in ~\ref{W_ij}, is an isometry and may not be necessarily a unitary. This forces us to take minimal unitary extension of the commuting tuple of isometries $(W_j^{(1)}, \ldots, W_j^{(j-1)}, W_j^{(j+1)}, \ldots, W_j^{(n)})$ defined on $\cld_{T_j}$. We denote the extended tuple of unitaries on $\clk_j$ again by $(W_j^{(1)}, \ldots, W_j^{(j-1)}, W_j^{(j+1)}, \ldots, W_j^{(n)})$. For $i=j$, we let the unitary $ W_j^{(j)}:\clk_j \to \clk_j$ to be  
\begin{align*}
    W_j^{(j)}:=W_j^{(1)*} \cdots W_j^{(j-1) *} W_j^{(j+1) *} \cdots W_j^{(n)*}. \end{align*}
Then, by the same arguments as before, the tuple $(W_j^{(1)},\dots, W_j^{(n)})$ is a tuple of commuting unitaries,
\begin{align*}
     W_j^{(j)}(D_{T_j} P_T^* h)=D_{T_j} T^*_j h\,\, (j=1,\dots, n, h\in\clh),
\text{ and }
    W_j^{(1)}\cdots W_j^{(n)}=I_{\cld_{T_j}}.
\end{align*}
The triple $(\cle,\ut{U},\ut{P})$ corresponding to the dilating BCL $n$-tuple is defined similarly by taking $\cle=\bigoplus^n_{i=1} \clk_i$, $\ut{U}=(U_1,\dots, U_n)$ and $\ut{P}=(P_1,\dots, P_n)$, where for $i=1,\dots,n$, 
\begin{align*}
    U^*_i=\bigoplus^n_{j=1} W_j^{(i)} \quad \text{and} \quad  P_i (x_1, \ldots, x_n)=(0,\ldots 0,  x_i , 0, \ldots 0)\quad (x_l \in \clk_l, l=1,\ldots,n).
\end{align*}
Since $\bigoplus^n_{i=1} \cld_{T_i}\subseteq \cle$, the isometry $\iota$, as in \ref{Iso_V}, can be viewed as an isometry from $\clh$ to $\cle$. Then by the same computation as done in the proof of the previous theorem, one also has 
\[\Pi T^*_i=M_{\Phi_i}^* \Pi\,\, (i=1,\ldots,n),\] where 
$\Pi:=(I \otimes \iota)\pi: \clh\to H^2_{\cle}(\D)$ and the isometry $\pi :\clh \to H^2_{\cld_{P_T}}(\D)$ is the dilation map of $P_T$ defined by $\pi h=\sum_{k \in\Z_+}  z^k (D_{P_T} P_T^{* k}h) $ ($z\in \D, h\in \clh$). Thus we have proved the following dilation result. 
 
\begin{Theorem} \label{Dila_infi}
Let $T=(T_1,\ldots, T_n)\in \mathfrak{B}^n_0(\clh)$. Then there exist a Hilbert space $\cle$, an isometry $\Pi: \clh \to H^2_\cle(\D)$, and a BCL $n$-tuple $(M_{\Phi_1}, \ldots,  M_{\Phi_n})\in \mathfrak{B}^n_0(H^2_{\cle}(\D))$ 
such that 
\[\Pi T^*_i=M_{\Phi_i}^* \Pi\]
for all $i=1,\ldots,n$.
\end{Theorem}
\begin{Remark}
For $T\in \mathfrak{B}^n_0(\clh)$, by the definition, the product contraction $P_T$ is pure. The above dilation result remains valid even if we remove the requirement that the product contraction $P_T$ is pure. In such a case, the isometric dilation we get would not be a BCL $n$-tuple. Instead, it would be direct sums of a BCL $n$ tuple and an $n$-tuple of commuting unitaries. Indeed, if 
\[
Q:=\text{SOT}-\lim_{n\to \infty} P_T^nP_T^{* n}\ \text{ and }
\clq= \overline{\text{ran}}\ Q,
\]
then for each $i=1,\dots,n$, the operator $\tilde{T_i}:\clq\to\clq$
defined by $\tilde{T_i}^*(Qh)= QT_i^* h$ is a co-isometry and $\tilde{T_i}\tilde{T_j}= \tilde{T_j}\tilde{T_i}$. Let $(W_1,\dots W_n)$ on $\clk$ be the minimal unitary co-extension of $(\tilde{T_1},\dots,\tilde{T_n})$. Since the minimal isometric dilation map of $P_T$ can be taken to be 
\[\tilde{\pi}:\clh\to H^2_{\cld_{P_T}}(\D)\oplus \clk, \quad h\mapsto (\pi(h),Q h)=(\sum_{k \in\Z_+}  z^k (D_{P_T} P_T^{* k}h), Qh),  
\]
one can check that 
\[
\Pi T_i^*=(M_{\Phi_i}^*\oplus W_i^*)\Pi \quad (i=1,\dots,n),
\]
where $(M_{\Phi_1},\dots M_{\Phi_n})\in \mathfrak{B}^n_0(H^2_{\cle}(\D))$ is the BCL $n$-tuple as constructed in Theorem~\ref{Dila_infi} and $\Pi:\clh \to H^2_{\cle}(\D)\oplus \clk$ is given by $\Pi(h)=((I\otimes \iota)\pi(h),Qh)$ with $\iota$ as in ~\eqref{Iso_V}.

\end{Remark}

 In the rest of this section, we prove a sharper version of von Neumann inequality. Since any element $T$ of $\mathfrak{B}^n_0(\clh)$ has an isometric dilation, $T$ satisfies the von Neumann inequality for $\mathbb{C}[z_1, \ldots, z_n]$, the ring of polynomials in $n$ variables with complex coefficients. That is, 
for all polynomials $p\in \mathbb{C}[z_1, \ldots, z_n]$,
\begin{align*}
    \|p(T)\|_{\clb(\clh)} \leq \sup_{\z\in \overline{\D}^n } |p(\z)|.
\end{align*} 
 However, we show that, in this particular case, one only needs to consider an one dimensional variety instead of whole $\overline{\D}^n$. Before going into further details, we need to recall some results. Let $T=(T_1,\ldots,T_n)$ be an $n$-tuple of commuting bounded operators on $\clh$. The Taylor joint spectrum (\cite{Taylor})  of $T$ is denoted by $\sigma(T)$. In the case when $T$ consists of commuting matrices, $\sigma(T)$ is the set of all joint eigenvalues. Let $(M_{\Phi_1}, \ldots, M_{\Phi_n})$ be a BCL $n$-tuple on $H^2_{\cle}(\D)$ such that $\cle$ is finite dimensional. Then, it is proved in \cite[Theorem 7.3]{BKS} that the set 
 \begin{align}\label{one_dim variety}
     W:=\{\z=(z_1,\ldots, z_n)\in\mathbb{C}^n: (z_1,\ldots, z_n)\in \sigma(\Phi_1(\underline{\z}),\ldots, \Phi_n(\underline{\z}))\},
 \end{align}
is a one dimensional symmetric algebraic variety, where for $\z=(z_1,\dots,z_n)\in\mathbb C^n$, $\underline{\z}:=z_1\cdots z_n$. Here the symmetric property of $W$ simply means that  
 \[(z_1,\ldots, z_n)\in W\quad \text{if and only if}\quad (\frac{1}{\bar{z}_1},\ldots, \frac{1}{\bar{z}_n})\in W.\]
We are now ready to prove the sharp von Neumann inequality. 
\begin{Theorem}\label{sharp vN inequality}
Let $T=(T_1,\ldots,T_n)\in \mathfrak{B}^n_0(\clh)$ and $\text{dim}\, \cld_{T_i}< \infty$ for all $i=1,\ldots,n$. Then there exists a one dimensional symmetric algebraic variety $W$ in $\mathbb{C}^n$ such that for all polynomials $p\in \mathbb{C}[z_1, \ldots, z_n]$,
\[\|p(T)\|_{\clb(\clh)} \leq \sup_{\z\in W\cap \overline{\D}^n} |p(\z)|.\]
\end{Theorem} 
 
\begin{proof}
Let the BCL $n$-tuple $(M_{\Phi_1}, \ldots, M_{\Phi_n})\in \mathfrak{B}^n_0(H^2_{\cle}(\D))$ be an isometric dilation of $T$ such that $\cle$ is finite dimensional. The existence of such a dilation is shown in Theorem \ref{Dilion_finite}. Moreover, it also follows from Theorem \ref{Dilion_finite} that  
\[T\cong (P_\clq M_{\Phi_{1}}|_\clq,\ldots, P_\clq M_{\Phi_{n}}|_\clq),\] 
where $\clq=\Pi(\clh)\subseteq H^2_{\cle}(\D)$ is a co-invariant subspace.
 Then, for $p\in\mathbb{C}[z_1, \ldots, z_n]$, 
\begin{align*}
    \|p(T)\|_{\clb(\clh)}&=\|P_\clq p( M_{\Phi_{1}},\ldots, M_{\Phi_{n}})|_\clq\|_{\clb(H^2_{\cle}(\D))}\\
    &\leq \| p( M_{\Phi_{1}},\ldots, M_{\Phi_{n}})\|_{\clb(H^2_{\cle}(\D))} \\ 
     &=\|M_{p(\Phi_{1}(z),\ldots, \Phi_{n}(z))}\|_{\clb(H^2_{\cle}(\D))}.
     \end{align*}
 Now, it is easy to observe that
 \[\|M_{p(\Phi_{1}(z),\ldots, \Phi_{n}(z))}\|_{\clb(H^2_{\cle}(\D))}= \|p(\Phi_{1}(z),\ldots, \Phi_{n}(z))\|_{H^\infty_{\clb(\cle)}(\D)}.\]
 Clearly, the right hand side is equal to
 \[\sup_{\theta\in [0,2\pi]}\|p(\Phi_1(e^{i\theta}),\ldots, \Phi_n(e^{i\theta}))\|_{\clb(\cle)}.\]
 Thus, we have 
 \[\|p(T)\|_{\clb(\clh)}\leq\sup_{\theta\in [0,2\pi]}\|p(\Phi_1(e^{i\theta}),\ldots, \Phi_n(e^{i\theta}))\|_{\clb(\cle)}.\]
Since $\text{dim}\,(\cle)<\infty$, for $z\in \mathbb{T}$, $(\Phi_1(z),\ldots, \Phi_n(z))$ is an $n$-tuple of commuting unitary matrices. Also since $\Phi_1(z)\cdots\Phi_n(z)=zI_{\cle}$ for all $z\in \overline{\D}$, it follows that for any $(z_1,\dots,z_n)\in \sigma(\Phi_1(z),\dots,\Phi_n(z))$,
\begin{align}\label{prod_z}
    \underline{\z}:=z_1\cdots z_n=z.
\end{align}
 Therefore,
\begin{align*}
   \sup_{z\in\mathbb T}\|p(\Phi_1(z),\ldots & , \Phi_n(z))\|_{\clb(\cle)} =\sup \{|p(z_1,\ldots,z_n)|: (z_1,\ldots,z_n)\in \sigma(\Phi_1(z),\ldots, \Phi_n(z)), z\in \mathbb T \}\\
    &=\sup \{|p(\z)|: \z=(z_1,\ldots,z_n)\in \mathbb T^n\cap  \sigma(\Phi_1(\underline{\z}),\ldots, \Phi_n(\underline{\z}))\} \, \, (\text{by }\eqref{prod_z})\\
   &= \sup_{\z\in W\cap {\mathbb T}^n} |p(\z)|\\
    & \leq \sup_{\z\in W\cap \overline{\D}^n} |p(\z)|,
\end{align*}
where $W$ is the one-dimensional symmetric algebraic variety as in ~\eqref{one_dim variety}. The above computation implies that $\sup_{\theta\in [0,2\pi]}\|p(\Phi_1(e^{i\theta}),\ldots, \Phi_n(e^{i\theta}))\|_{\clb(\cle)}\leq \sup_{\z\in W\cap \overline{\D}^n} |p(\z)|$. This completes the proof. 
\end{proof}


\newsection{Commutant lifting theorem for $\mathfrak{B}^n_0(\clh)$ }\label{Sec4}

For $T=(T_1,\ldots, T_n)\in \clt^n(\clh)$, we denote by $\{T\}^\prime$ the set of all bounded operators on $\clh$ which commute with each $T_i$ ($i=1,\ldots,n$), that is, 
\[\{T\}^\prime:=\{X\in\clb(\clh): X T_i=T_i X, i=1,\ldots,n\}.\]
The foundation of the commutant lifting theorem for $\mathfrak{B}^n_0(\clh)$ is the explicit isometric dilation of an element $T\in \mathfrak{B}^n_0(\clh)$ to a BCL $n$-tuple $M\in \mathfrak{B}^n_0(H^2_{\cle}(\D))$ which we obtained in the previous section. As a result, our proof of the commutant lifting theorem hinges on the detailed analysis of the structure of $\{T\}^\prime$ as well as of $\{M\}^{\prime}$, where $T\in \mathfrak{B}^n_0(\clh)$ and $M\in \mathfrak{B}^n_0(H^2_{\cle}(\D))$ is a BCL $n$-tuple. To begin with, we consider the case of BCL $n$-tuple.

\begin{Proposition}\label{Form of commutant} 
Let $M=(M_{\Phi_1},\ldots, M_{\Phi_n})\in\mathfrak{B}^n_0(H^2_{\cle}(\D))$ be a BCL $n$-tuple corresponding to $(\cle, \ut{U}, \ut{P})$. Suppose that $\cle_i=\text{ran}\,P_i $ for all $i=1,\dots,n$. If $X$ is a contraction in $\{M\}^\prime$, then $X$ has the following block matrix form with respect to decomposition $H^2_{\cle}(\D)=H^2_{\cle_1}(\D)\oplus\cdots \oplus H^2_{\cle_n}(\D)$:
\begin{align}\label{Matrix_form}
    X=\begin{bmatrix}
M_{\Theta_1} & 0&\cdots&0\\
0 & M_{\Theta_2}&\cdots&0\\
\vdots &\vdots& \ddots&\vdots\\
0&0& \cdots&M_{\Theta_n}\end{bmatrix},
\end{align}
 where for all $i=1,\dots,n$, $\Theta_i$ is a contractive $\clb(\cle_i)$-valued analytic function on $\D$.
\end{Proposition}
\begin{proof}
For $X\in \{M\}^\prime$, since $X M_z=M_z X$, then there exists $\Theta \in H_{\clb(\cle)}^\infty (\D)$ such that $X=M_{\Theta}$. Let $\Theta(z)=\sum_{ k \geq  0}A_k z^k$, where $A_k \in \clb(\cle)$ ($k \geq 0$).
We fix $i\in \{1,\dots,n\}$. Since $X M_{\Phi_i}=M_{\Phi_i} X$, then the identity $ \Phi_i(z) \Theta(z)=\Theta(z) \Phi_i(z)$ implies that for all $z\in \D$,
 \begin{align*}
    U_i P_i^\perp A_0 +\sum_{ k \geq  1}( U_i P_i^\perp A_k+ U_i P_i A_{k-1}) z^{k} = A_0  U_i P_i^\perp  +\sum_{ k \geq  1}(A_k U_i P_i^\perp +A_{k-1} U_i P_i) z^{k}. 
\end{align*}
 Comparing the coefficients, we obtain
 \begin{align}\label{Power_coffe}
      U_i P_i^\perp A_0=A_0  U_i P^\perp_i\,\, \text{and}\,\, U_i P_i^\perp A_k+ U_i P_i A_{k-1}=A_k U_i P_i^\perp +A_{k-1} U_i P_i \,\, (k\geq 1).
 \end{align}
Note that as $M\in \mathfrak{B}^n_0(H^2_{\cle}(\D))$, we have $U_i P_j=P_j U_i$ for all $i,j=1,\dots,n$. Then the identity $U_i P_i^\perp A_0=A_0  U_i P^\perp_i$ together with $A^*_0 P_i^\perp U_i^* = P^\perp_iU_i^*A_0^*$, obtained by taking adjoint, imply that $\cle_i^\perp$ is a reducing subspace for $A_0$. That is, $A_0 P_i^{\perp}=P_i^{\perp}A_0$.  Letting $k=1$ in \ref{Power_coffe} we get \[
U_i P_i^\perp A_1+ U_i P_i A_0=A_1 U_i P_i^\perp +A_0 U_i P_i.\]
 Consequently, 
 \begin{align*}
   U_i P^\perp_i A_1P^\perp_i= (U_i P_i^\perp A_1+ U_i P_i A_0)P_i^\perp=(A_1 U_i P_i^\perp +A_0 U_i P_i)P_i^\perp
     =A_1 U_i P^\perp_i
     \end{align*}
By the same argument as before, the above identity together with $P_i^\perp A^*_1 P^\perp_i U^*_i=P^\perp_i U^*_i A_1^*$ imply that $\cle_i^\perp$ is a reducing subspace for $A_1$. Proceeding in the same way, we can show that $\cle_i^\perp$ is a reducing subspace for $A_k$ for all $k\ge 2$. We set $A^{(i)}_k:=A_k|_{\cle_i}$ for all $k\ge 0$, and $\Theta_i(z):=\sum_{k\geq0}A^{(i)}_k z^k$. Then $\Theta_i\in H_{\clb(\cle_i)}^\infty (\D)$. Since $i$ was arbitrarily chosen, we have $\Theta_i\in H_{\clb(\cle_i)}^\infty (\D)$ for all $i=1,\dots,n$ and it is now easy to see that
\[
\Theta(z)= \oplus_{i=1}^n \Theta_i(z), \text{ and therefore } M_{\Theta}=\oplus_{i=1}^n M_{\Theta_i}.
\]
This completes the proof. 
\end{proof}
As a consequence of the above result, we find certain decomposition of the defect operator next. 
\begin{Proposition} 
Let $M=(M_{\Phi_1},\ldots, M_{\Phi_n})\in\mathfrak{B}^n_0(H^2_{\cle}(\D))$ be a BCL $n$-tuple corresponding to $(\cle,\ut{U}, \ut{P})$. If $X$ is a contraction in $\{M\}^\prime$, then there exist positive operators $G_1, \ldots, G_n$ on $H^2_{\cle}(\D)$ such that 
\begin{align*}
  I-X X^*=G_1 +\cdots + G_n\quad \text{and} \quad  (I-C_{M_{\Phi_j}})(G_i)=G_i -M_{\Phi_j} G_i M^*_{\Phi_j} = 0 \quad (i\neq j).
\end{align*}   
\end{Proposition}
\begin{proof}
Since $X$ is a contraction in $\{M\}^\prime$, using Proposition \ref{Form of commutant}, $X$ has the block matrix from as in \ref{Matrix_form} for some $\Theta_i \in H_{\clb(\cle_i)}^\infty (\D)$ $(i=1,\dots,n)$.
We consider,
    \[G_i=\begin{bmatrix}
0 & \dots &0&\dots&0\\
\vdots &\ddots &\vdots &\ddots &\vdots \\
0 &\dots & I_{H^2_{\cle_i}(\D)}- M_{\Theta_i}M^*_{\Theta_i}&\dots &0\\
\vdots  &\ddots & \vdots &\ddots &\vdots\\
0 &\dots &0 &\dots &  0
\end{bmatrix},\quad (i=1,\dots,n)\]
where the block matrix representation is taken with respect to the decomposition $H^2_{\cle}(\D)=H^2_{\cle_1}(\D)\oplus\cdots \oplus H^2_{\cle_n}(\D)$. 
 Since $M_{\Theta_i}$ is a contraction, it follows that $G_i\ge 0$ and 
 $I-XX^*=G_1+\dots+G_n$. Now for $i\neq j$, using the block matrix form \ref{BCL_Matrix} of $M_{\Phi_j}$ and $M_{\Phi_j}X=XM_{\Phi_j}$, it can be easily verified that $M_{\Phi_j}G_iM_{\Phi_j}^*=G_i$. The proof now follows. 
 \end{proof}

Let $T=(T_1,\ldots, T_d)\in \mathfrak{B}^n_0(\clh)$ and $M=(M_{\Phi_1}, \ldots,  M_{\Phi_n})\in \mathfrak{B}^n_0(H^2_{\cle}(\D))$ be an isometric dilation of $T$. That is, there exist an isometry $\Pi: \clh \to H^2_\cle(\D)$ such that 
\[\Pi T^*_i=M_{\Phi_i}^* \Pi \quad (i=1,\ldots,n).\] For a contractive $X\in \{T\}^\prime$, if there exists a contractive multiplier $\Theta \in H^\infty_{\clb(\cle)}(\D)$ such that $M_{\Theta}\in \{M\}^\prime$ and $\Pi X^*= M^*_{\Theta} \Pi$, then by the above proposition $I-M_{\Theta}M_{\Theta}^*= G_1+\dots+G_n$ for some positive operators $G_i$ on $H^2_{\cle}(\D)$ such that $(I-C_{M_{\Phi_j}})(G_i)=G_i -M_{\Phi_j} G_i M^*_{\Phi_j} = 0 \quad (i\neq j)$. Then observe that 
\begin{align*}
    I-X X^*= \Pi^*(I-M_\Theta M^*_{\Theta})\Pi 
    =\Pi^* G_1 \Pi +\cdots +\Pi^* G_n \Pi
\end{align*}
and 
\begin{align*}
    (I-C_{T_j})(\Pi^* G_i \Pi)= \Pi^*  (I-C_{M_{\Phi_j}})(G_i)\Pi=0\quad (i\neq j).
\end{align*}
Thus, the condition that there should exist positive operators $\tilde{G_1},\dots \tilde{G_n}$ on $\clh$ for which 
\[
I-XX^*=\tilde{G_1}+\cdots +\tilde{G_n}, \text{ and }  (I-C_{T_j})(\tilde{G_i})=0,
\quad (i\neq j)\]
is necessary for the commutant lifting theorem to hold. In fact, we show below that every commutant of $T\in \mathfrak{B}^n_0(\clh)$ indeed satisfies the necessary condition, for which we need the following lemma.

\begin{Lemma}\label{Lemma_ness}
Let $T=(T_1,\ldots, T_n)\in \clt^n(\clh)$ be such that $ \Delta^{\bm{e}_i+\bm{e}_j}_{T}=0$ $(i\neq j)$. Then $I-P_T P_T^* =\sum_{i=1}^n D^2_{T_i}$, where $P_T=T_1\cdots T_n$ and $D_{T_i}$ is the defect operator of $T_i$.
\end{Lemma}
\begin{proof}
 Using Agler's hereditary functional calculus (\cite{AG1}) to the following identity, for $(z_1, \ldots, z_n)$ and $(w_1, \ldots, w_n) \in \D^n$, 
\begin{align*}
     (1-z_1\Bar{w}_1 \cdots z_n \Bar{w}_n)&=\sum_{i=1}^n (1-z_i \Bar{w}_i)-\sum^n_{\underset{i< j}{i,j=1}}(1-z_i\Bar{w}_i)(1-z_j\Bar{w}_j)\\
    & \quad \quad \quad - \cdots (-1)^{n-1} (1-z_1\Bar{w}_1)\cdots (1-z_n \Bar{w}_n),
\end{align*}
 we get
\begin{align*}
       I -P_T P_T^*=\sum_{i=1}^n D^2_{T_i}-\sum^n_{\underset{i< j}{i,j=1}} \Delta^{\bm{e}_i+\bm{e}_j}_{T}+
    \sum_{\underset{i< j< k}{i,j,k =1}}^n \Delta^{\bm{e}_i+\bm{e}_j+\bm{e}_k}_{T}- \cdots (-1)^{n-1} \Delta^{\bm{e}_i+\cdots+\bm{e}_n}_{T}.
\end{align*}
Hence, by Lemma \ref{Breh_0}, $I-P_T P_T^* =\sum_{j=1}^n D^2_{T_j}$.
\end{proof}

For $T=(T_1,\ldots,T_n)\in\clt^n(\clh)$, we denote by $\Tilde{T}_j$ the $(n-1)$-tuple obtained from $T$ by deleting $T_j$, that is, 
\[\Tilde{T}_j:=(T_1,\ldots, T_{j-1},  T_{j+1}, \ldots, T_n).\]
\begin{Theorem}\label{Nessry_cond}
Let $T=(T_1,\ldots, T_n)\in \mathfrak{B}^n_0(\clh)$. If $X$ is a contraction in $\{T\}^\prime$, then there exist positive operators $G_1, \ldots, G_n$ on $\clh$ such that 
\begin{align}\label{BLTT}
  I-X X^*=G_1 +\cdots + G_n\,\, \text{and }  (I-C_{T_i})(G_j)=G_j -T_i G_j T^*_i = 0 \quad (i\neq j).
\end{align}  
\end{Theorem}
\begin{proof}
First we define
\[G_j=\text{SOT}-\lim_{\alpha=(\alpha_1,\dots,\alpha_{n-1})\to \infty} \Tilde{T}_j^\alpha (I-XX^*)\Tilde{T}_j^{*\alpha}\quad (j=1,\ldots, n).\]
Clearly, the above SOT limit exists as $T_i$'s are contractions. 
It then follows from the definition that $(I-C_{T_i})(G_j)=G_j -T_i G_j T^*_i = 0$ for all $i\neq j$.
Now, 
\begin{align*}
    \sum_{j=1}^n G_j=\text{SOT}-\lim_{\alpha\to \infty} \sum_{j=1}^n \Tilde{T}_j^\alpha (I-XX^*)\Tilde{T}_j^{*\alpha}
    =\text{SOT}-\lim_{\alpha\to \infty} \big( \sum_{j=1}^n \Tilde{T}_j^\alpha \Tilde{T}_j^{*\alpha}-X(\sum_{j=1}^n \Tilde{T}_j^\alpha\Tilde{T}_j^{*\alpha})X^*\big). 
    \end{align*}
To complete the proof, it is enough to show that $\text{SOT}-\lim_{\alpha\to \infty} \sum_{j=1}^n \Tilde{T}_j^\alpha \Tilde{T}_j^{*\alpha}=I_{\clh}$. Since $\Delta^{\bm{e}_i+\bm{e}_j}_{T}=0$, that is, $D_{T_j}^2=T_i D^2_{T_j}T^*_i$ ($i\neq j$), then
\[D_{T_j}^2=\Tilde{T}_j^\alpha D^2_{T_j} \Tilde{T}_j^{*\alpha}\quad (j=1,\ldots,n, \alpha=(\alpha_1, \ldots, \alpha_{n-1})\in \Z^{n-1}_+).
\]
Taking sum over $j$ in both sides to the above identity, we get for all $\alpha\in \Z^{n-1}_+$
\begin{align*}
   \sum_{j=1}^n \Tilde{T}_j^\alpha\Tilde{T}_j^{*\alpha}& =\sum_{j=1}^n D^2_{T_j}+  \sum_{j=1}^n \Tilde{T}_j^\alpha T_j T^*_j \Tilde{T}_j^{*\alpha}\\
   &=I_{\clh}-P_T P_T^* + \sum_{j=1}^n \Tilde{T}_j^\alpha T_j T^*_j \Tilde{T}_j^{*\alpha}\quad(\text{by Lemma  } \ref{Lemma_ness}).
   \end{align*}
Consequently, for $\alpha\ge \mathbf e=(1,\dots, 1)$,
\begin{align*}
  I_{\clh}-\sum_{j=1}^n \Tilde{T}_j^\alpha\Tilde{T}_j^{*\alpha}=
  P_T(I_{\clh}  - \sum_{j=1}^n \Tilde{T}_j^{\alpha-e} \Tilde{T}_j^{*{\alpha-e}}) P_T^*.
\end{align*}
Set $Y_{\alpha}:= I_{\clh}-\sum_{j=1}^n \Tilde{T}_j^\alpha\Tilde{T}_j^{*\alpha}, (\alpha\in \Z^{n-1}_+)$. By the above identity, for $\alpha\ge \mathbf e$, if $k=\text{min}\{\alpha_1,\dots,\alpha_{n-1}\}$, then 
\begin{align*}
Y_{\alpha}= P_T Y_{\alpha-\mathbf e}P_T^*=\cdots=P^k_{T} Y_{\alpha-k\mathbf e}P_T^{* k}.
\end{align*} 
Thus, for $h\in\clh$,
\begin{align*}
    \|Y_\alpha h\|=\|P_T^k Y_{\alpha-ke} P_T^{*k} h\|\leq \|Y_{\alpha-ke}\| \|P_T^{*k} h\|\leq (n+1)\|P_T^{*k} h\|.
\end{align*}
Since $P_T$ is pure, it then follows that $\text{SOT}-\lim_{\alpha \to \infty} Y_\alpha=0$.
 This completes the proof.
\end{proof}
The identity \eqref{BLTT}, in the above theorem, is crucial for commutant lifting theorem to hold, which is one of the main results of this section.
The reader familiar with the work of ~\cite{BLTT} could easily recognize the identity \eqref{BLTT}, 
and must be tempted to use ~\cite{BLTT} to prove the commutant lifting theorem. Despite of the similarity, one can not use ~\cite{BLTT} directly as the isometric dilation of $T$ as well as the dilation map are completely different. As a result, our proof is entirely different compared to that of ~\cite{BLTT}.     
Before going into the commutant lifting theorem, we need couple of preparatory results. The first one of them is of independent interest. It is well known that given any contractive $\clb(\cle)$-valued analytic function $\Theta$, there exist an auxiliary Hilbert space $\clf$ and a contraction
\[
W=\begin{bmatrix}
A& B\\
C& D
\end{bmatrix}: \cle\oplus\clf \to \cle\oplus \clf
\]
such that $\Theta(z)=\tau_{W}(z)$, where $\tau_{W}$ is the transfer function corresponding to $W$, that is, 
\[
\tau_{W}(z)=A+ zB(I_{\clf}-zD)^{-1}C\quad (z\in\D). 
\]
Moreover, the contraction $W$ can be taken to be a unitary if necessary. 
This result is known as transfer function realizations for operator valued bounded analytic functions on the unit disc.
Recall that a contraction $T$ on $\clh$ is said to be a \emph{completely non-coisometric} if there is no non-trivial invariant subspace of $T^*$ on which $T^*$ is an isometry. A contraction $T$ is said to be a \emph{completely non-unitary} (cnu) if there is no non-zero $T$-reducing subspace $\cls$ of $\clh$ such that $T|_{\cls}$ is a unitary operator.  For a contraction $T$ on $\clh$, there exist unique $T$-reducing subspaces $\clh_{cnu}$ and $\clh_{u}$ such that $\clh= \clh_{cnu} \oplus \clh_{u}$, $T_{cnu}:=T|_{\clh_{cnu}}$ is a cnu and $T_{u}:=T|_{\clh_{u}}$ is a unitary (see \cite[Theorem 3.2]{NF}). Thus, $T$ has the following block matrix representation of $T$ with respect to the decomposition $\clh_{cnu}\oplus\clh_{u}$: 
\[
T=\begin{bmatrix}T_{cnu}& 0\\
0& T_{u}
\end{bmatrix}.
\]
Such a decomposition is commonly refer as canonical decomposition of $T$.
If $U\in\clb(\cle)$ is an isometry, in the result below, we find a necessary and sufficient condition for the commutativity of $\tau_{W}(z)$ and $U$ for all $z\in\D$ in terms of $W$ and $U$.

\begin{Theorem}\label{comm_lifting}
Let $\tau_{W}(z)=A+zB(I-z D)^{-1} C$ be the transfer function corresponding to a unitary  $W=\begin{bmatrix}
A& B\\
C& D
\end{bmatrix}$ on $\cle \oplus \clf$ for some Hilbert spaces $\cle$ and $\clf$, and let $U$ be an isometry on $\cle$. Let
\[
D=\begin{bmatrix}D_{1}& 0\\
0& D_{2}
\end{bmatrix}
\]
be the canonical decomposition of $D$ into the cnu part $D_1$ and the unitary part $D_2$ with respect to the decomposition $\clf=\clf_1\oplus \clf_2$.
Consider the following statements.
\begin{enumerate}
    \item[\textup{(1)}] $M_{\tau_{W}} (I_{H^2(\D)}\otimes U)=(I_{H^2(\D)}\otimes U) M_{\tau_{W}}$.
    \item[\textup{(2)}] There exists a contraction $Y$ on $\clf$ such that $W \Tilde{U}=\Tilde{U} W$, where $\Tilde{U}=\begin{bmatrix}
U& 0\\
0& Y
\end{bmatrix}$. 
\end{enumerate}
Then $(2)\Longrightarrow (1)$. Moreover, if $D_1$ is completely non-coisometric,
then $(1) \Longrightarrow (2)$.
\end{Theorem}

\begin{proof}
Using the power series expansion of $\tau_{W}(z)$, it is easy to see that the identity in $(1)$ 
is equivalent to 
\begin{align}\label{equivalent identities}
AU=UA \text{ and } (B D^{k-1} C)U=U(B D^{k-1} C)\quad  (k\geq 1). 
\end{align}
Now we assume that $(2)$ holds. 
Then by a routine block matrix multiplication,
$AU=U A$, $BY=UB$, $CU=YC$, and $DY=YD$. 
Thus, $AU=U A$ and  
\begin{align*}
 B D^{k-1} CU= B D^{k-1} YC=B Y D^{k-1} C=U B D^{k-1} C \quad (k\geq 1).  
\end{align*}
Hence, by the observation we made earlier, $M_{\tau_W} (I_{H^2(\D)}\otimes U)=(I_{H^2(\D)}\otimes U) M_{\tau_W}$. This proves $(2) \Longrightarrow (1)$. 

For the proof of the moreover part 
note that 
the block matrix representation of $W$ with respect to the decomposition $\cle\oplus \clf_1\oplus\clf_2$ is 
\[
W=\begin{bmatrix}
A & B|_{\clf_1}& 0\\
P_{\clf_1}C & D_1 & 0\\
0& 0 & D_2
\end{bmatrix}.
\]
Since $W$ is a unitary,
\begin{align*}
     A^* A+ C^* P_{\clf_1} C=I_\cle,  \quad P_{\clf_1} B^* B|_{\clf_1} +D_1^* D_1=I_{\clf_1},
\end{align*}
\[P_{\clf_1} B^* A +D_1^* P_{\clf_1} C=0, \text{ and } P_{\clf_1} C C^*|_{\clf_1}+ D_1 D_1^*=I_{\clf_1}.\]
\textbf{Claim:} 
\[
\bigvee_{k\geq 0} D_1^k P_{\clf_1} C (\cle)= \clf_1.
\]
The proof of the claim follows from \cite[Theorem 4.2]{BallCohen}. However, we supply a proof for the completeness. To this end, it is enough to show that $\cls:=(\clf_1 \ominus \bigvee_{k\geq 0} D_1^k P_{\clf_1} C (\cle))=\{0\}$. 
Let $h\in \cls$, that is, $h\in \clf_1$ and $h\perp \bigvee_{k\geq 0} D_1^k P_{\clf_1} C (\cle)$. Therefore, $h \perp D_1^k P_{\clf_1} C \eta$ ($k\geq 0, \eta\in \cle$). In particular, $h \perp P_{\clf_1} C \eta$ ($\eta\in \cle$), that is,
\[\langle h, P_{\clf_1} C \eta \rangle=0 \quad (\eta \in \cle).\]
This implies $h\in \text{ker }C^*$. In other words, $\cls\subseteq \text{ker }C^*$. Now it is easy to see that $\cls$ is a $D^*_1$-invariant subspace and the identity $P_{\clf_1} C C^*+ D_1 D_1^*=I_\clf$ restricted on $\cls$ becomes $D_1 D_1^*|_{\cls}=I_\cls$, which contradicts that $D_1$ is completely non-coisometric. Hence the claim is proved. 


We now proceed to construct a contraction $Y$ on $\clf$ such that 
$W\tilde{U}=\tilde{U}W$, that is, 
\[BY= UB, CU=YC, \text{ and } D Y=YD. \]
Since $\clf_1=\bigvee_{k\geq 0} D_1^k P_{\clf_1} C (\cle)$, we define 
\[ Y:\clf_1 \oplus \clf_2 \to \clf_1 \oplus \clf_2,\text{ by } (D_1^k P_{\clf_1} C \eta, f) \mapsto (D_1^k P_{\clf_1} C U \eta, 0)\quad (\eta\in \cle, f\in \clf_2, k\geq 0).\] 
It remains to show that $Y$ is a well-defined contraction on $\clf$ and $BY=UB$ as the identities $CU=YC$ and $D Y=YD$ are immediate from the definition of $Y$. In fact, the reader must have noticed that it is the only way one could define $Y$ with required properties, and therefore $Y$ is unique in this case.
We show below that $Y|_{\clf_1}$ is a unitary. To this end, it is enough to show that 
\[\langle D_1^k P_{\clf_1} C \eta, D_1^l P_{\clf_1} C \eta^\prime\rangle =\langle D_1^k P_{\clf_1} C U \eta, D_1^l P_{\clf_1} C U \eta^\prime \rangle\]
for all $k, l\in\Z_+$ and $\eta, \eta^\prime\in \cle$.
When $l=k=0$, note that  
\begin{align*}
\langle P_{\clf_1} C \eta, P_{\clf_1} C \eta^\prime \rangle &= \langle \eta,(I_\cle-A^* A) \eta^\prime  \rangle \quad(\text{as }A^* A+ C^* P_{\clf_1} C=I_\cle)\\
    &=\langle U\eta,U\eta^\prime \rangle- \langle UA \eta, UA \eta^\prime \rangle\quad (\text{as } U \text{ is an isometry})\\
    &=\langle U \eta, U \eta^\prime \rangle- \langle AU \eta,  A U \eta^\prime \rangle \quad (\text{as } AU =UA)\\
   &= \langle P_{\clf_1} CU \eta, P_{\clf_1} CU \eta^\prime\rangle. 
\end{align*}
For the general case we use the following identity repeatedly
\begin{align}\label{general identity}
\langle D_1^m P_{\clf_1}  C \eta, D_1^n  P_{\clf_1} C \eta^\prime \rangle&=\langle D_1^{m-1}P_{\clf_1} C \eta, D_1^*D_1D_1^{n-1} P_{\clf_1} C \eta^\prime \rangle\nonumber\\
&= \langle D_{1}^{m-1} P_{\clf_1} C\eta, D_{1} ^{n-1}  P_{\clf_1} C \eta^\prime \rangle- \langle B D_{1} ^{m-1} P_{\clf_1} C\eta, B D_{1}^{n-1} P_{\clf_1} C \eta^\prime \rangle,
\end{align}
%
as $P_{\clf_1} B^* B|_{\clf_1}+ D_1^* D_1=I_{\clf_1}$,
where $m,n\ge 1$.
Now consider the case when at least one of $l$ and $k$ is non-zero. We assume, without any loss of generality, that $k\geq l$. Then for any $\eta, \eta^\prime\in \cle$, by applying \eqref{general identity} $l$ times, we have
\begin{align*}
    \langle D_1^k P_{\clf_1} C \eta, D_1^l P_{\clf_1} C \eta^\prime \rangle
   &=\langle D_1^{k-l} P_{\clf_1} C\eta, P_{\clf_1} C \eta^\prime \rangle- \sum^{l}_{j=1} \langle B D_1^{k-j} P_{\clf_1} C\eta, B D_1^{l-j} P_{\clf_1} C \eta^\prime \rangle \nonumber\\
   &=\langle D_1^{k-l-1} P_{\clf_1} C \eta, D_1^* P_{\clf_1} C \eta^\prime \rangle- \sum^{l}_{j=1} \langle B D_1^{k-j} P_{\clf_1} C\eta, B D_1^{l-j} P_{\clf_1} C \eta^\prime \rangle\nonumber\\
    &=-\langle D_1^{k-l-1} P_{\clf_1} C\eta, B^* A \eta^\prime \rangle- \sum^{l}_{j=1} \langle B D_1^{k-j} P_{\clf_1} C\eta, B D_1^{l-j} P_{\clf_1} C \eta^\prime \rangle \nonumber\\
    &=-\langle B D_1^{k-l-1} P_{\clf_1} C\eta,  A \eta^\prime \rangle- \sum^{l}_{j=1} \langle B D_1^{k-j} P_{\clf_1} C\eta, B D_1^{l-j} P_{\clf_1} C \eta^\prime \rangle.
\end{align*}
In the above identity, replacing $\eta$ and $\eta^\prime$ by $U\eta$ and $U\eta^\prime$, respectively and using \eqref{equivalent identities}, we get
 \begin{align*}   
 \langle D_1^k P_{\clf_1} C U\eta, D_1^l P_{\clf_1} C U\eta^\prime \rangle& =-\langle B D_1^{k-l-1} P_{\clf_1} CU\eta,  A U\eta^\prime \rangle- \sum^{l}_{j=1} \langle  B D_1^{k-j} P_{\clf_1} CU\eta,  B D_1^{l-j} P_{\clf_1} C U\eta^\prime \rangle\\ 
 &=-\langle U B D_1^{k-l-1} P_{\clf_1} C\eta,  UA \eta^\prime \rangle- \sum^{l}_{j=1} \langle U B D_1^{k-j} P_{\clf_1} C\eta, U B D_1^{l-j} P_{\clf_1} C \eta^\prime \rangle\\
     &=-\langle  B D_1^{k-l-1} P_{\clf_1} C \eta,  A \eta^\prime \rangle- \sum^{l}_{j=1} \langle  B D_1^{k-j} P_{\clf_1} C \eta,  B D_1^{l-j} P_{\clf_1} C  \eta^\prime \rangle\\
     &= \langle D_1^k P_{\clf_1} C \eta, D_1^l P_{\clf_1} C  \eta^\prime \rangle.
\end{align*}
Thus $Y$ is a unitary on $\clf_1$. From the definition of $Y$, it follows that $Y$ is a contraction.
 Now for $g=\sum_{k=0}^lD_1^k P_{\clf_1} C \eta_k$ ($\eta_k\in \cle$) and $f\in \clf_2$,
\begin{align*}
    B Y (g,f)=B Y \big(\sum_{k=0}^lD_1^k P_{\clf_1} C \eta_k, f \big)&= \big(\sum_{k=0}^l B D_1^k P_{\clf_1} C U \eta_k, 0\big)\\
    & = UB  \big (\sum_{k=0}^l  D_1^k P_{\clf_1} C \eta_k, f\big )\\
    &= U B(g,f). 
\end{align*}
Therefore, $BY =UB$ and this completes the proof.
\end{proof}

The other result that we need appeared in different contexts before (see \cite[Lemma 4.1]{MB} and \cite[Lemma 2.1]{DS}). According to our need, we present a slightly modified version.  
\begin{Lemma}\label{Lifting_id}
Let $(S,T,X)$ be a triple of commuting contractions on $\clh$ such that $T$ is pure, and let $G$ and $Q$ be two contractions on $\clh$ such that $\overline{\text{ran}}\, G\subseteq \clg$ and  $\overline{\text{ran}}\, Q\subseteq \clk$ for some Hilbert spaces $\clg$ and $\clk$. Assume that $\Pi: \clh \to H^2_{\clk}(\D)$, defined by $h\mapsto \sum_{k\ge 0}z^k Q S^* T^{* k} h$, is a bounded operator and
\[W=\begin{bmatrix}
A&B\\
C&D
\end{bmatrix}:\clk \oplus \clg \to\clk \oplus \clg \] 
is a contraction satisfying 
\[ W(Qh, G T^* h)=(Q X^* h, G h)\quad (h\in \clh).\]
Then the transfer function $\Theta(z)=\tau_{W^*}(z)=A^*+C^*z (I-z D^*)^{-1} B^*$ $(z \in \D)$, corresponding to the contraction $W^*$,  satisfies 
\[\Pi X^* = M^*_{\Theta} \Pi.\]
 \end{Lemma} 

\begin{proof}
Since
\begin{align*}
    \begin{bmatrix}
A& B\\
C& D
\end{bmatrix}\begin{bmatrix}
Qh\\
G T^* h\end{bmatrix}=\begin{bmatrix}
QX^* h\\
G h
\end{bmatrix},
\end{align*}
we have 
 \begin{align}\label{Equ_1}
     A Q h+ B G T^* h=QX^* h, 
      \end{align}
and
 \begin{align}\label{Equ_2}
      C Qh+ D G T^* h=G h.
      \end{align}
Now, replacing $h$ by $T^* h$ in \ref{Equ_2}, we have 
 \begin{align*}
      G T^* h=C Q T^*h+ D G T^{* 2} h.
      \end{align*}
    Then, substituting $G T^*h$ by $C Q T^*h+ D G T^{* 2} h$  in \ref{Equ_1}, we obtain
\begin{align*}
      QX^* h =A Qh+ BC Q T^*h+ B D G T^{* 2} h. 
      \end{align*}
 By repeating the procedure $m$ times, we get     
\begin{align*}
     Q X^* h=A Qh+\sum^{m-1}_{k=0} B D^{ k} C Q T^{* k+1}h+ B D^m G T^{* m+1} h.
 \end{align*}
Since $T$ is pure, for all $h\in\clh$, $\|B D^m G T^{* m+1} h\|\leq \|T^{* m+1} h\|\to 0$ as $m\to \infty$. Therefore,
\begin{align}\label{Id_for_inter}
   Q X^*=A Q+\sum^{\infty}_{k=0} B D^{ k} C Q T^{* k+1}.
 \end{align}
Now, for $ k\in \Z_+, \eta\in\cle, h\in \clh$,
\begin{align*}
    \langle M_{\Theta}^* \Pi h, z^k \eta\rangle &= \langle \sum_{m\ge 0}z^m Q S^* T^{* m}h,  (A^*+C^*z (I-z D^*)^{-1} B^*)z^k \eta\rangle\\
    &= \langle \sum_{m\ge 0}z^m Q S^* T^{* m}h,  (A^*+ \sum_{m\ge 0}z^{m+1 }C^* D^{* m} B^*)z^k \eta\rangle\\
     &= \langle A Q S^* T^{* k}h, \eta \rangle + \sum_{m\ge 0} \langle B D^ m C Q S^* T^{* m+k+1}h, \eta\rangle\\
    &= \langle \big(A Q  + \sum_{m\ge 0} B D^ m C Q  T^{* m+1}\big) S^* T^{* k}h, \eta\rangle\\
     &= \langle Q  X^* S^* T^{* k}h, \eta\rangle \quad (\text{by the identity \ref{Id_for_inter}}),
\end{align*}
and on the other hand,
\begin{align*}
    \langle  \Pi X^* h, z^k \eta\rangle = \langle \sum_{m\ge 0}z^m Q S^* T^{* m} X^* h,  z^k \eta\rangle=\langle Q  X^* S^* T^{* k } h,  \eta\rangle.
\end{align*}
Hence, $\Pi X^* = M^*_{\Theta} \Pi$. This completes the proof.
\end{proof}

For $\alpha=(\alpha_1,\ldots, \alpha_n)\in \Z^n$, we define $\alpha_+:=(\alpha_1^+,\ldots, \alpha^+_n)\in\Z^n_+$ and $\alpha_-:=(\alpha_1^-,\ldots, \alpha^-_n)\in\Z^n_+$, where $\alpha_i^+=\text{max}\,\{\alpha_i, 0\}$ and $\alpha_i^-=\text{max}\,\{-\alpha_i, 0\}$ ($i=1,\ldots,n$). We are now ready to prove the commutant lifting theorem. 
\begin{Theorem}\label{Comm_lift_theo}
Let $T=(T_1,\ldots, T_n)\in \mathfrak{B}^n_0(\clh)$.  Suppose that the isometric dilation of $T$ is given by the BCL $n$-tuple $M=(M_{\Phi_1}, \ldots,  M_{\Phi_n})\in \mathfrak{B}^n_0(H^2_{\cle}(\D))$ via the dilation map $\Pi: \clh \to H^2_\cle(\D)$ for some Hilbert space $\cle$, as in Theorem \ref{Dila_infi}.      %
If $X$ is a contraction in $\{T\}^\prime$, then there exists a contractive multiplier $\Theta \in H^\infty_{\clb(\cle)}(\D)$ such that $M_{\Theta}\in \{M\}^\prime$ and $\Pi X^*= M^*_{\Theta} \Pi$. 
\end{Theorem}
\begin{proof} 
Since the proof uses the explicit isometric dilation of $T\in \mathfrak{B}^n_0(\clh)$ as constructed in Theorem \ref{Dila_infi}, we first recall some of the necessary ingredients from the proof of Theorem \ref{Dila_infi}.
By Theorem \ref{Dila_infi}, the Hilbert space $\cle$ has the decomposition $\cle=\oplus^n_{j=1}\clk_j$, where $\cld_{T_j}\subseteq \clk_j$ for all $j=1,\dots,n$. The dilation map $\Pi: \clh \to H^2_\cle(\D)$ is an isometry given by 
\[
\Pi h=\sum^\infty_{k=0}  z^k (D_{T_1} P_T^{* k} h, D_{T_2} T^*_1 P_T^{* k} h, \ldots, D_{T_n}T^*_1 \cdots T^*_{n-1} P_T^{* k} h)\quad (h\in\clh).
\]
The dilating BCL $n$-tuple $M$ corresponds to the triple $(\cle,\ut{U},\ut{P})$ where $\ut{U}=(U_1,\dots, U_n)$ and $\ut{P}=(P_1,\dots, P_n)$ are defined as 
\begin{align}
    U^*_i=\bigoplus^n_{j=1} W_j^{(i)}\,\,  \text{and }  P_i (x_1, \ldots, x_n)=(0,\ldots 0,  x_i , 0, \ldots 0)\quad (x_l \in \clk_l, l=1,\ldots,n),
\end{align} 
for some unitaries $W_j^{(i)}: \clk_j \to \clk_j$ satisfying
 \begin{align}\label{Wij}
     W_j^{(i)}(D_{T_j}h)=D_{T_j} T^*_i h\quad ( i\neq j) \text{ and }
     W_j^{(j)}(D_{T_j} P_T^* h)=D_{T_j} T^*_j h 
 \end{align}
 for all $h\in\clh$.
Moreover, for all $j=1,\dots,n$, 
\begin{align}\label{Equ_Wjj}
    W_j^{(j)}=W_j^{(1)*} \cdots W_j^{(j-1) *} W_j^{(j+1) *} \cdots W_j^{(n)*}.
\end{align}
For all $j=1,\ldots,n$, we define 
\[
\Pi_j:\clh \to H^2_{\clk_j}(\D) \text{ by }  h\mapsto \sum^\infty_{k=0}  z^k  D_{T_j}T^*_1 \cdots T^*_{j-1} P_T^{* k}h.
\]
It is now clear that for all $h\in \clh$, $\Pi h=\oplus_{j=1}^n \Pi_jh $.
 Suppose any contractive multiplier $\Theta \in H^\infty_{\clb(\cle)}(\D)$ is in $\{M\}^\prime$. Then, by Proposition \ref{Form of commutant}, $M_\Theta$ has the form $M_{\Theta}=\oplus_{j=1}^n M_{\Theta_j}$, where  $\Theta_j\in H_{\clb(\clk_j)}^\infty (\D)$ for all $j=1,\dots,n$.
 By the explicit form of $M_{\Phi_i}$, as in \ref{BCL_Matrix}, $M_{\Theta}$ commutes with $M_{\Phi_i}$ for all $i=1,\dots,n$ if and only if 
\begin{equation}\label{NS 2}
\Theta_j(z)W_j^{(i)}=W_j^{(i)}\Theta_j(z)\quad (z\in\D, i,j=1,\dots,n).
\end{equation}
 Moreover, $\Pi X^*= M^*_{\Theta} \Pi$ if and only if for all $h\in \clh$
 \begin{equation}\label{NS 1}
 \bigoplus_{j=1}^n \Pi_j X^*h=\bigoplus_{j=1}^n M^*_{\Theta_j} \Pi_jh, \text{ equivalently } \Pi_j X^*=M^*_{\Theta_j} \Pi_j\quad (j=1,\dots,n).
 \end{equation}
Thus by the above discussion, to complete the proof, we need to construct 
$\Theta_j\in H_{\clb(\clk_j)}^\infty (\D)$ for all $j=1,\dots,n$ such that ~\eqref{NS 1} and ~\eqref{NS 2} holds. To this end, for the rest of the proof, we fix $j$ and find an auxiliary Hilbert space $\clg$ and a contraction $R$ on $\clk_j\oplus \clg$ such that $\Theta=\tau_{R^*}$, the transfer function corresponding to $R^*$. Also note that for $\tau_{R^*}$ to satisfy equation ~\eqref{NS 2}, by Theorem~\ref{comm_lifting}, it is sufficient to find contractions $Y_i$ on $\clg$ such that 
\begin{align}\label{Comm_RV}
    R^* \begin{bmatrix}
 W_j^{(i)}& 0\\
 0&Y_i
 \end{bmatrix}= \begin{bmatrix}
 W_j^{(i)}& 0\\
 0&Y_i
 \end{bmatrix} R^* \quad (i=1,\dots,n).
\end{align}
  This is the point of departure to the construction of the space $\clg$ and contractions $R$ and $Y_i$ with required properties. 
 Since $X$ is a contraction in $\{T\}^\prime$, by Theorem \ref{Nessry_cond}, there exist positive operators $G_1, \ldots, G_n$ on $\clh$ such that 
\begin{align}\label{Id_nesry}
  I-X X^*=G_1 +\cdots + G_n \text{ and }  (I-C_{T_i})(G_m)=G_m -T_i G_m T^*_i = 0 \quad (i\neq m).
\end{align} 
Repeated application of the second identity above yields $T_jG_jT_j^*=P_T G_j P_T^*$. Now, 
\begin{align*}
   D^2_{T_j}-X D^2_{T_j} X^*&=I-T_j T_j^* -X(I-T_j T^*_j)X^*\\
   &=(I-C_{T_j})(I-X X^*)\\
   &=(I-C_{T_j})(G_1+\cdots+ G_n)\\
   &=G_j-T_j G_j T^*_j \quad (\text{by \ref{Id_nesry}}).
\end{align*}
A rearrangement of the above identity is 
\begin{align}\label{Id_TX}
    D^2_{T_j}+P_T G_j P_T^*=X D^2_{T_j} X^*+G_j.
\end{align}
We set $\tilde{\clg}:=\overline{\text{ran}}\, G_j^{1/2}$,
 \[\clm:=\{(D_{T_j}h, G_j^{1/2}P_T^*h):h\in\clh\}\subseteq \clk_j\oplus \tilde{\clg}, \text{ and } \cln:=\{(D_{T_j}X^*h, G_j^{1/2}h):h\in\clh\}\subseteq \clk_j\oplus \tilde{\clg}.\]
 Therefore, the identity \ref{Id_TX} induces an isometry $\tilde{R}:\clm \to \cln$, 
 defined by 
 \begin{align}\label{tilde_R}
      \tilde{R}(D_{T_j}h, G^{1/2}_j P_T^* h)=(D_{T_j}X^* h, G^{1/2}_j h)\quad (h\in \clh).
 \end{align}
  Again, using the identity $G_j=T_i G_j T^*_i$ ($i \neq j$), we define an isometry for each $i\neq j$,
\begin{align}\label{Map_Y_i}
    \Tilde{Y_i}:\Tilde{\clg} \to \Tilde{\clg}, \text{ by } G_j^{1/2}h\mapsto G_j^{1/2} T^*_i h\quad (h\in \clh).
\end{align}
Since $(\Tilde{Y}_1, \ldots, \Tilde{Y}_{j-1}, \Tilde{Y}_{j+1}, \ldots, \Tilde{Y}_n)$ is an $(n-1)$-tuple of commuting isometries on $\Tilde{\clg}$, it has minimal unitary extension $(Y_1, \ldots, Y_{j-1}, Y_{j+1}, \ldots, Y_n)$ on $\clg$, say.  We set 
\begin{align}\label{Equ_Yj}
    Y_j:=Y_1^* \cdots Y^*_{j-1} Y^*_{j+1} \cdots Y^*_n.
\end{align}
For the simplicity of notations, we set 
\[
V_i:=\begin{bmatrix}
 W_j^{(i)}& 0\\
 0&Y_i
 \end{bmatrix}\quad (i=1,\dots,n).
\]
The auxiliary space $\clg$ and the required contractions $Y_i$'s on $\clg$ are already constructed. To construct the contraction $R$, we take extension of $\tilde{R}$, as in ~\eqref{tilde_R}, in certain way so that ~\eqref{Comm_RV} holds. It follows from ~\eqref{Wij} and ~\eqref{Map_Y_i} that $\clm$ and $\cln$ are invariant subspaces of each of the unitary  $V_i$ on $\clk_j\oplus \clg$
 with $i\neq j$. Also note that for $h\in\clh$ and $i\neq j$,
\begin{align*}
    \Tilde{R}V_i(D_{T_j}h, G_j^{1/2}P_T^*h)=\Tilde{R} (D_{T_j}T^*_i h, G_j^{1/2}T^*_i P_T^*h)= (D_{T_j}T^*_i X^*h, G_j^{1/2}T^*_ih),
\end{align*}
and 
on the other hand 
\begin{align*}
   V_i \Tilde{R} (D_{T_j}h, G_j^{1/2}P_T^*h)= V_i (D_{T_j}X^* h, G_j^{1/2}h)
    =(D_{T_j}T^*_i X^*h, G_j^{1/2}T^*_ih).
\end{align*}
Thus, $ \Tilde{R} V_i|_\clm=V_i \Tilde{R}|_{\clm}$ for all $i\neq j$.
We consider  
\[\Tilde{\clm}:=\bigvee_{\alpha\in \Z^{n-1}_+}\Tilde{V}^{* \alpha}(\clm) \text{ and } \Tilde{\cln}:=\bigvee_{\alpha\in \Z^{n-1}_+}\Tilde{V}^{* \alpha}(\cln),\]
where $\Tilde{V}:=(V_1, \ldots, V_{j-1}, V_{j+1}, \ldots,  V_n)$,
and extend $\tilde{R}$ to get an operator from $\Tilde{\clm}$ onto $\Tilde{\cln}$ by the following prescription 
\[
\sum_{0\leq \alpha\leq \gamma}\Tilde{V}^{* \alpha} x_\alpha\mapsto \sum_{0\le \alpha\leq \gamma}\Tilde{V}^{* \alpha} \Tilde{R} x_\alpha\quad (x_\alpha\in \clm).
\]
The extension is well-defined as 
\begin{align*}
    \| \sum_{0\le \alpha\leq \gamma}\Tilde{V}^{* \alpha} \Tilde{R} x_\alpha\|^2
&=\sum_{0\le \alpha, \beta\le \gamma}\langle \Tilde{V}^{* \alpha} \Tilde{R} x_\alpha, \Tilde{V}^{* \beta} \Tilde{R} x_\beta \rangle\\
&=\sum_{0\le \alpha, \beta\le \gamma} \langle \Tilde{V}^{ (\beta -\alpha)_+} \Tilde{R} x_\alpha, \Tilde{V}^{ (\beta -\alpha)_-} \Tilde{R} x_\beta \rangle\\
&=\sum_{0\le \alpha,\beta\le \gamma}\langle \Tilde{R}\Tilde{V}^{ (\beta -\alpha)_+}  x_\alpha, \Tilde{R}\Tilde{V}^{ (\beta-\alpha)_-}  x_\beta \rangle\quad (\text{as } \Tilde{R} V_i|_\clm=V_i \Tilde{R}|_{\clm}, i\neq j)\\
&=\sum_{0\le \alpha,\beta\le \gamma}\langle \Tilde{V}^{ (\beta -\alpha)_+}  x_\alpha, \Tilde{V}^{ (\beta-\alpha)_-}  x_\beta \rangle\\
&= \| \sum_{0\le \alpha\leq \gamma}\Tilde{V}^{* \alpha} x_\alpha\|^2.
\end{align*}

We denoted the extended operator again by $\Tilde{R}: \tilde{\clm}\to\tilde{\cln}$ and it follows from the definition that 
\[\tilde{R} V_i|_{\Tilde{\clm}}= V_i\tilde{R}|_{\Tilde{\clm}}\quad (i\neq j).\]
Finally, we define the contraction $R$ on $\clk_j\oplus \clg$ by setting  $R= \tilde{R}$ on $\tilde{\clm}$ and $R\equiv 0$ on $\tilde{\clm}^{\perp}$. Since $\tilde{\clm}$ is a reducing subspace for $V_i$ $(i\neq j)$, we have $RV_i=V_iR$ for all $i\neq j$. Also since 
$V_j=V_1^*\cdots V_{j-1}^*V_{j+1}^*\cdots V_n^*$, which can be read off from ~\eqref{Equ_Wjj} and \eqref{Equ_Yj}, it follows that $RV_j=V_jR$. Then by Fuglede-Putnam, we have ~\eqref{Comm_RV}. It only remains to prove that
\[\Pi_j X^*=M^*_{\Theta_j}\Pi_j,\]
where $\Theta_j(z)=\tau_{R^*}(z)=A^*+C^*z (I-z D^*)^{-1} B^*$, transfer function of $R^*$ and
\[R=\begin{bmatrix}
A& B\\
C& D
\end{bmatrix}: \clk_j \oplus \clg\to \clk_j \oplus \clg.\]
This is obtained by applying Lemma~\ref{Lifting_id} for $S:=T_1\cdots T_{j-1}$, $T:= P_T$, $G:=G_j^{1/2}$, $Q:=D_{T_j}$, $\Pi:=\Pi_j$, $W:=R$, and using ~\ref{tilde_R}. Hence the theorem is proved.  
\end{proof}

\begin{Remark}
For $T=(T_1,\dots,T_n)\in \mathfrak{B}^n_0(\clh)$ and contractive $X$ in $\{T\}'$, the above theorem shows that $X$ has a contractive lifting in the isometric dilation space of $T$. This in particular implies that one can find a norm preserving lifting of $X$. Indeed, if $\|X\|=1$ then there is nothing to prove. If $0<\|X\| <1$, and if $Y$ is a contractive lifting of $\frac{X}{\|X\|}$ then $\|Y\|=1$ and clearly $\|X\|Y$ is a norm preserving lifting of $X$.    
\end{Remark}

The following result is a particular case of ~\cite[Proposition 9.2]{NF}, we include a proof for the sake of completeness.  
\begin{Proposition}\label{Lifting_pro}
Let $(V_1,\ldots, V_n)$ be an $n$-tuple of commuting isometries on $\clh$, and let $X$ be a contraction on $\clh$ such that $X V_i=V_i X$ for all $i=1,\ldots,n$. Then the tuple $(X, V_1, \ldots, V_n)$ has an isometric dilation on the minimal isometric dilation space of $X$. 
\end{Proposition}

\begin{proof}
Since $X$ is a contraction, suppose $Y$ on $\clk$ is the minimal isometric dilation of $X$. That is,
\[X^k=P_\clh Y^k|_\clh \, (k\in \Z_+) \text{ and } \clk=\bigvee_{k\in \Z_+} Y^k(\clh).\]
We define, for each $i=1,\ldots,n$, $\Tilde{V_i}:\clk \to \clk$ by
\[\Tilde{V}_i(\sum^m_{k=0} Y^k h_k)=\sum^m_{k=0} Y^k V_i h_k \quad (h_k\in \clh).\]
It is now enough to prove that each $\Tilde{V_i}$ is well-defined, which readily follows from the following computation. For any $h, h^\prime\in\clh$ and $k, l\in \Z_+$ with $k\geq l$,
\begin{align*}
    \langle Y^k V_i h, Y^l V_i h^\prime \rangle= \langle Y^{k-l} V_i h, V_i h^\prime \rangle=\langle X^{k-l} V_i h, V_i h^\prime \rangle=\langle V_i X^{k-l}  h, V_i h^\prime \rangle=\langle  X^{k-l}  h, h^\prime \rangle.
\end{align*}
This completes the proof. 
\end{proof}

As a consequence of Theorem \ref{Comm_lift_theo} and Proposition \ref{Lifting_pro}, we have the following dilation result. 
 \begin{Theorem}\label{main 2}
Let $(T_1,\dots,T_n, T_{n+1})$ be an $(n+1)$-tuple of commuting contractions on $\clh$ such that $(T_1,\dots,T_n)\in \mathfrak{B}^n_0(\clh)$. Then $(T_1,\dots,T_n, T_{n+1})$ has an isometric dilation and consequently, satisfies the von Neumann inequality. 
\end{Theorem}

\newsection{Examples}\label{Sec5}
 In this section, we exhibit structures of tuples in $\mathfrak{B}^n_0(\clh)$, when $\clh$ is finite dimensional. Also we find an example to demonstrate that such a structure theorem does not hold for the general case of infinite dimensional Hilbert space.  We have pointed out earlier that if $(T_1,T_2)$ is a pair of commuting contractions on $\clh$ such that $T_1$ is a co-isometry and $T_2$ is a pure contraction, then $T_1T_2$ is a pure contraction and $\Delta^{(1,1)}_{(T_1,T_2)}=0$, and therefore $(T_1, T_2)\in \mathfrak{B}^2_0(\clh)$. In fact, in the next proposition, we show that pureness of one of the contraction forces the other operators to be co-isometries and vice versa.

\begin{Proposition}\label{pure vs co-isometry}
Let $T=(T_1,\ldots, T_n)\in \mathfrak{B}^n_0(\clh)$. Then $T_i$ is a pure contraction if and only if $T_j$ is  a co-isometry for all $j\neq i$. 
\end{Proposition}
\begin{proof}
Without of loss generality, let $T_1$ be a pure contraction. Since $(T_1,\ldots, T_n)\in \mathfrak{B}^n_0(\clh)$, then, for $i=2, \ldots, n$, $\Delta^{(1,1)}_{(T_1,T_i)}=D^2_{T_i}-T_1 D^2_{T_i} T^*_1=0$, where $D^2_{T_i}=I-T_i T^*_i$. Iterating the identity $k$-times, we have
\[D^2_{T_i}=T^k_1 D^2_{T_i} T^{*k}_1\quad (k\in \Z_+). 
\]
Using pureness of $T_1$, we obtain $D^2_{T_i}=0$, and therefore, $T_i$ is a co-isometry. Conversely, let $T_2, \ldots, T_n $ are  co-isometries. Since $P_T$ is a pure contraction, for $h\in\clh$,
\[\|T^{* k}_1 h\|=\|T^{k}_2 \cdots T^{k}_n  T^{* k}_n \cdots T^{* k}_2 T^{*k}_1 h\|\leq \| P^{* k}_T h\| \to 0\text{ as } n\to \infty.\]
This completes the proof.  
\end{proof}


It turns out that if $\clh$ is finite dimensional that any $n$-tuple $(T_1, \ldots, T_n)$ in $\mathfrak{B}^n_0(\clh)$ arises in this way. We first prove it for $n=2$, as it is much simpler, and subsequently we treat the general case. 
We denote by $\mathbb{M}_m(\mathbb{C})$ the space of all $m\times m$ complex matrices.


\begin{Proposition}\label{finite}
Let $(A,B)\in \mathfrak{B}^2_0(\mathbb{C}^m)$. Then there exist
commuting pairs $(U_1, T_1)$ in $\mathbb{M}_k(\mathbb{C})$ and 
$(T_2, U_2)$ in $\mathbb{M}_{m-k}(\mathbb{C})$, where $k\le m$, $U_1$ and $U_2$ are unitary, and $T_1$ and $T_2$ are matrices with no unimodular eigenvalues such that  
\[(A, B)\cong \left(\begin{bmatrix} U_1 & 0\\
0 &  T_2\end{bmatrix}, \begin{bmatrix} T_1 & 0\\
0 &  U_2\end{bmatrix}\right).\]
\end{Proposition}
\begin{proof}
Since $A$ and $B$ are commuting matrices, then the pair $(A^*,B^*)$ can be simultaneously upper tringularizable. Suppose that 
\[
(A^*,B^*)\cong \left( \begin{bmatrix}
a_{11}& \cdots & \cdots& a_{1m}\\
0& a_{22}& \cdots & a_{2 m}\\
\vdots & \vdots & \ddots & \vdots\\
0 &0 &  \cdots & a_{mm}
\end{bmatrix}, \begin{bmatrix}
b_{11}& \cdots & \cdots& b_{1m}\\
0& b_{22}& \cdots & b_{2 m}\\
\vdots & \vdots & \ddots & \vdots\\
0 &0 &  \cdots & b_{mm}
\end{bmatrix} \right).
\]
Then the joint eigenvalues of $A^*$ and $B^*$ are given by $\{(a_{ii}, b_{ii}): i=1,\dots,m\}$. Let $v_i$ be joint eigenvector of norm one corresponding to the joint eigenvalues $(a_{ii}, b_{ii})$. Then 
\begin{align*}
0=\langle \Delta^{(1,1)}_{(A,B)}v_i, v_i\rangle =\langle (I-AA^*-BB^*+ABA^*B^*)v_i, v_i\rangle=(1-|a_{ii}|^2)(1-|b_{ii}|^2).
\end{align*}
This shows that either $|a_{ii}|=1$ or $|b_{ii}|=1$. Also note that both $a_{ii}$ and $b_{ii}$ can not have modulus one, as $A^{*n}B^{* n}\to 0$ in SOT as $n\to \infty$. Now the result follows by rearranging all the unimodular eigenvalues of $A^*$ and $B^*$.  
\end{proof}
 Now we prove the result for any $n$-tuple ($n\geq 3$) of matrices in $\mathfrak{B}^n_0({\mathbb{C}^m})$. Let $A=(A_1, \ldots, A_n)\in \mathfrak{B}^n_0(\mathbb{C}^m)$. Since $(A_1, \ldots, A_n)$ is a tuple of commuting  matrices, the tuple $(A_1^*,\ldots, A_n^*)$ is simultaneously upper tringularizable. Let us assume that  
\[
(A_1^*,\ldots, A_n^*)\cong \left( \begin{bmatrix}
a^{(1)}_{11}& \cdots & \cdots& a^{(1)}_{1m}\\
0& a^{(1)}_{22}& \cdots & a^{(1)}_{2 m}\\
\vdots & \vdots & \ddots & \vdots\\
0 &0 &  \cdots & a^{(1)}_{mm}
\end{bmatrix},\ldots, \begin{bmatrix}
a^{(n)}_{11}& \cdots & \cdots& a^{(n)}_{1m}\\
0& a^{(n)}_{22}& \cdots & a^{(n)}_{2 m}\\
\vdots & \vdots & \ddots & \vdots\\
0 &0 &  \cdots & a^{(n)}_{mm}
\end{bmatrix} \right).
\]
Let us fix $i$ ($1\le i\le m$). Since $A_1^{*k}\cdots A_n^{* k}\to 0$ in SOT as $k\to \infty$ and $(a^{(1)}_{ii},\ldots, a^{(n)}_{ii})$ is a joint eigenvalues of $(A_1^*,\ldots, A_n^*)$, $a_{ii}^{(k)}\in \D$ for some $1\le k\le n$.
On the other hand, for $l\neq j$, a similar computation as in the proof of Proposition \ref{finite} reveals that
 \[0=\langle \Delta^{\bm{e}_{l}+\bm{e}_{j}}_{A}v_i, v_i\rangle =(1-|a^{(l)}_{ii}|^2)(1-|a^{(j)}_{ii}|^2) \quad (l\neq j),\]
 where $v_i$ is a joint eigenvector of norm one corresponding to the joint eigenvalue $(a^{(1)}_{ii},\ldots, a^{(n)}_{ii})$ of the tuple $(A^*_1, \ldots, A^*_n)$. The above identity implies that, $a_{ii}^{(k)}\in \D$ for exactly one $k\in \{1,\dots,n\}$ and the rest are unimodular.  
Since this is true for all $i=1, \ldots, m$, it leads to the following result.   

\begin{Proposition}\label{n_tuple_finite}
Let $A=(A_1, \ldots, A_n)\in \mathfrak{B}^n_0(\mathbb{C}^m)$. Set $k=\min \{n,m\}$. Then a rearrangement of the tuple $A$ is unitarily equivalent to a commuting tuple $(B_1, \ldots, B_k, W_1, \ldots, W_{n-k})$, where $B_j$'s and $W_i$'s are block diagonal matrices corresponding to the decomposition $\mathbb C^m=\mathbb C^{m_1}\oplus\cdots\oplus \mathbb C^{m_k}$ and have the form 
\[B_j=\begin{bmatrix} U^{(1)}_j \\
&\ddots \\
&  & T_j&\\
& & & \ddots &\\
& & & & U^{(k)}_j
\end{bmatrix}  \quad \text{and} \quad W_i= \begin{bmatrix}  V^{(1)}_i\\
&\ddots \\
&  & \ddots&\\
& & & \ddots &\\
& & & & V^{(k)}_i
\end{bmatrix},\]
for some matrices $T_j$'s in $\mathbb{M}_{m_j}(\mathbb{C})$ with no unimodular eigenvalues and for unitaries $U^{(l)}_j$ $(l\neq j)$ and $V^{(l)}_i$ in $\mathbb{M}_{m_l}(\mathbb{C})$.
\end{Proposition}

It is worth mentioning here that a pair of commuting isometries $(V_1,V_2)\in\mathfrak{B}^2_0(\clh) $ also has the form as in Proposition\ref{finite}. Indeed, since $V_1V_2$ is pure then $(V_1,V_2)$ is unitarily equivalent to a BCL pair 
$(M_{\Phi_1}, M_{\Phi_2})\in \mathfrak{B}^2_0(H^2_{\cle}(\D))$ for some co-efficient Hilbert space $\cle$. Then it follows from Proposition~\ref{main_Characterization} and ~\eqref{BCL_Matrix} that $(M_{\Phi_1}, M_{\Phi_2})$ has the required form. However, Proposition~\ref{finite} is not true for arbitrary elements in $\mathfrak{B}^2_0(\clh)$ when $\clh$ is infinite dimensional. We demonstrate this in the following example.


\textbf{Example:} 
Let $B$ be the bilateral shift on $l^2(\mathbb{Z})$, defined by
\[B(\{\ldots, a_{-2},a_{-1}, \bold{a_0}, a_1, a_2, \ldots\})=\{\ldots, a_{-3},a_{-2}, \bold{a_{-1}}, a_0, a_1, \ldots\}.\] 
Set $\cle:=l^2(\mathbb{Z})\oplus l^2(\mathbb{Z})$. Consider the unitary $U:=B\oplus B$ on $\cle$ and define a projection 
\[P:\cle \to \cle, \text{ by } P((x,y))=(x,0)\quad (x,y\in l^2(\mathbb{Z})).
\]
Let $(M_\Phi, M_\Psi)$ on $H^2_{\cle}(\D)$ be the BCL pair corresponding to $U$ and $P$, that is, 
\[\Phi(z)=(P +zP^\perp)U^*\quad \text{and}\quad \Psi(z)=U(P^\perp+z P)\quad (z\in \D).\]
Since $UP=PU$, it follows from Proposition~\ref{main_Characterization} that  $(M_\Phi, M_\Psi)\in \mathfrak{B}^2_0(H^2_{\cle}(\D))$. 
Moreover, with respect to the decomposition $H^2_{l^2(\Z)}(\D)\oplus H^2_{l^2(\Z)}(\D)$,
 \begin{align}\label{Form matrix M_phi}
     M_{\Phi}=\begin{bmatrix}
I\otimes B^* &0\\
0& M_z \otimes B^*
\end{bmatrix} \text{ and } M_{\Psi}=\begin{bmatrix}
M_z\otimes B &0\\
0& I \otimes B
\end{bmatrix}.
 \end{align}
For a fixed $\lambda_0\in\D$, let us define a subspace $\clq$ of $H^2_{\cle}(\D)$ as follows 
\[\clq:=\bigvee_{k, l\geq 0}M^{* k}_\Phi M^{* l}_\Psi (K_{\lambda_0}\otimes (e_0, e_0)),\]
where $e_0=\{\ldots, 0,\bold{1}, 0,\ldots\}\in l^2(\mathbb{Z})$ and $K_{\lambda_0}$ is the kernel function at $\lambda_0$, that is, 
\[K_{\lambda_0}(z)=\frac{1}{(1-\bar{\lambda}_0z)}\quad (z\in\D).\] 
It is easy to see that $\clq$ is joint $(M_\Phi^*, M_\Psi^*)$ invariant subspace. We define
\[T_1=P_\clq M_\Phi|_\clq \text{ and } T_2=P_\clq M_\Psi|_\clq.\] 
Then $(T_1, T_2)$ is a pair of commuting contractions on $\clq$ with $T_1T_2= P_{\clq}M_z|_{\clq}$. Moreover, since $\Delta^{(1,1)}_{(M_{\phi}. M_{\psi})}=0$, we have $\Delta^{(1,1)}_{(T_1. T_2)}=0$. Therefore, $(T_1,T_2)\in \mathfrak{B}^2_0(\clq)$. We show below that the pair $(T_1, T_2)$ does not have the form as in Proposition~\ref{finite}. We divide the arguments in two cases.

\textbf{Case I:} We claim that neither $T_1$ nor $T_2$ is a co-isometry, equivalently, by Proposition~\ref{pure vs co-isometry} neither $T_1$ nor $T_2$ is a pure contraction. Without any loss of generality, suppose that $T_1$ is a pure contraction. 
Since $K_{\lambda_0}\otimes (e_0, e_0)\in \clq$, for $n\in \mathbb N$,
\begin{align*}
    \|T_1^{* n}\big(K_{\lambda_0}\otimes (e_0, e_0)\big)\|^2&=\|M^{*n}_\Phi \big(K_{\lambda_0}\otimes (e_0, e_0)\big)\|^2\\
    & =\left\|\begin{bmatrix}
I\otimes B^{n} &0\\
0& M^{*n}_z \otimes B^{n}
\end{bmatrix}\begin{pmatrix}
K_{\lambda_0}\otimes e_0\\
K_{\lambda_0}\otimes e_0
\end{pmatrix}\right\|^2\\
&=\|K_{\lambda_0}\otimes B^ne_0\|^2+
\|M^{*n}_z K_{\lambda_0}\otimes B^{n} e_0
\|^2.
\end{align*}
This shows that $\|T_1^{* n}\big(K_{\lambda_0}\otimes (e_0, e_0)\big)\|$ does not converges to $0$ as $n\to \infty$, which is a contradiction. 

\textbf{Case II:} We show that there is no joint $(T_1, T_2)$-reducing subspace $\clq_1$ of $\clq$ such that both $T_1|_{\clq_1}$ and $T_2|_{\clq_2}$ are co-isometry, where $\clq=\clq_1 \oplus \clq_2$. Suppose on the contrary that there exists a joint $(T_1, T_2)$-reducing subspace $\clq_1$ of $\clq$ so that both $T_1|_{\clq_1}$ and $T_2|_{\clq_2}$ are co-isometry. Since an arbitrary element of $\clq$ is of the form $K_{\lambda_0}\otimes (x',y')$ for some $x', y' \in l^2(\mathbb{Z})$ and $\clq_1\subseteq\clq$, suppose that $K_{\lambda_0}\otimes (x, y)\in \clq_1$. We claim that $y=0$. Indeed, for any $n\in\mathbb N$,  
\begin{align*}
   \|K_{\lambda_0}\otimes (x, y)\|^2&= \|T_1^{* n}\big(K_{\lambda_0}\otimes (x, y)\big)\|^2\\
   &=\|M^{*n}_\Phi \big(K_{\lambda_0}\otimes (x, y)\big)\|^2\\
     &=\left\|\begin{bmatrix}
I\otimes B^{n} &0\\
0& M^{*n}_z \otimes B^{n}
\end{bmatrix}\begin{pmatrix}
K_{\lambda_0}\otimes x\\
K_{\lambda_0}\otimes  y
\end{pmatrix}\right\|^2\\
&=\|K_{\lambda_0}\otimes B^nx\|^2+
\|M^{*n}_z K_{\lambda_0}\otimes B^n y
\|^2.
\end{align*}
Since $M_z$ is a pure isometry, by taking $n\to \infty$ in the above identity we have $\| (x, y)\|^2=\| x\|^2$, and therefore $y=0$. Thus we have shown that 
 $\clq_1 \subseteq H^2_{l^2(\mathbb{Z})\oplus \{0\}}$. By a similar set of arguments, along with the fact that $T_2|_{\clq_2}$ is a co-isometry, we also have $\clq_2 \subseteq H^2_{\{0\}\oplus l^2(\mathbb{Z})}$. Since $\clq=\clq_1\oplus\clq_2$ and $K_{\lambda_0}\otimes (e_0, e_0)\in \clq$, then  $K_{\lambda_0}\otimes (e_0, 0)\in \clq_1$. 
Next we claim that $\clq$ has the following orthogonal decomposition 
\[\clq=\text{Span}\,\{K_{\lambda_0}\otimes (e_0, e_0)\}\oplus \bigvee_{k\neq l}M^{* k}_\Phi M^{* l}_\Psi K_{\lambda_0}\otimes (e_0, e_0).\]  
Indeed, note that for $k=l$, $M_{\Phi}^{* k}M_{\Psi}^{* l}K_{\lambda_0}\otimes (e_0, e_0) = M_z^* K_{\lambda_0}\otimes (e_0, e_0)=  K_{\lambda_0}\otimes \Bar{\lambda}_0(e_0, e_0)$. Thus, we only need to show the orthogonality. Without any loss of generality suppose that $k>l$, then 
\begin{align*}
\langle M_{\Phi}^{* k}M_{\Psi}^{* l}K_{\lambda_0}\otimes (e_0, e_0), K_{\lambda_0}\otimes (e_0, e_0) \rangle & = \langle M_{\Phi}^{* (k-l)} M_z^{* l}
K_{\lambda_0}\otimes (e_0, e_0), K_{\lambda_0}\otimes (e_0, e_0) \rangle\\
& = \langle K_{\lambda_0} \otimes( B^{(k-l)} e_0, \overline{\lambda}_0^{(k-l)}  B^{(k-l)} e_0),
K_{\lambda_0}\otimes \lambda_0^l(e_0, e_0) \rangle\\
&=0.\end{align*}
Here we have used the fact that for $k\in \Z_+$,  $M_\Phi^{* k}K_{\lambda_0}\otimes (e_0, e_0)=K_{\lambda_0} \otimes( B^k e_0, \overline{\lambda}_0^k  B^k e_0)$, which can be computed using \ref{Form matrix M_phi}. This proves the claim. Also by the same computation as we have done above, for $l\neq k$, $\langle M^{* k}_\Phi M^{* l}_\Psi K_{\lambda_0}\otimes (e_0, e_0), K_{\lambda_0}\otimes (e_0, 0) \rangle=0$. Therefore, $K_{\lambda_0}\otimes (e_0, 0)\in \text{Span}\,\{K_{\lambda_0}\otimes (e_0, e_0)\}$, which is a contradiction. 

\vspace{0.1in} \noindent\textbf{Conflict of interest:}
The author states that there is no conflict of interest. No data sets were generated or analyzed during the current study.

\vspace{0.1in} \noindent\textbf{Acknowledgement:}
We thank to the anonymous referee for many helpful suggestions.  The first author is supported by the Mathematical Research Impact Centric Support (MATRICS) grant, File No: MTR/2021/000560, by the Science and Engineering Research Board (SERB), Department of Science \& Technology (DST), Government of India.

\bibliographystyle{amsplain}

\end{document}